\documentclass[11pt]{amsart}
\usepackage{amsmath,amssymb,latexsym,mathrsfs}
\usepackage[left=2.7cm,right=2.7cm,top=3cm,bottom=3cm]{geometry}

\usepackage{color,graphicx}
\usepackage[colorlinks=true,urlcolor=blue, citecolor=red,linkcolor=blue,
linktocpage,pdfpagelabels, bookmarksnumbered,bookmarksopen]{hyperref}
\usepackage[hyperpageref]{backref}
\usepackage[english]{babel}
\newtheorem{theorem}{Theorem}[section]
\newtheorem{proposition}[theorem]{Proposition}
\newtheorem{lemma}[theorem]{Lemma}

\newcommand{\R}{\mathbb{R}}

\newcommand{\eps}{\varepsilon}
\newcommand{\beq}{\begin{equation}}
\newcommand{\eeq}{\end{equation}}

\def\XXint#1#2#3{{\setbox0=\hbox{$#1{#2#3}{\int}$ }
\vcenter{\hbox{$#2#3$ }}\kern-.6\wd0}}

\newcommand{\io}{\int_{\Omega}}

\newcommand{\n}{D}

\newcommand{\vf}{\varphi}
\newcommand{\Om}{\Omega}

\newcommand{\f}{\Phi}

\DeclareMathOperator{\Div}{div}

\title[Finsler $p$-Laplace type equations]{Concave solutions to Finsler $p$-Laplace type equations}

\author[S. Mosconi]{Sunra Mosconi}

\author[G. Riey]{Giuseppe Riey}

\author[M. Squassina]{Marco Squassina}

\thanks{2010 Mathematics Subject Classification: 35J92, 35B33, 35B06}
\keywords{Nonsmooth analysis, anisotropic problems, convexity of solutions}

\thanks{\rm
All authors are members of the GNAMPA group of the Istituto Nazionale di Alta Matematica (INdAM) and they are supported by it.
The first author is supported by GNAMPA project CUP E53C22001930001, PRIN project 2022ZXZTN2 and PIACERI line 2 and 3.
The second author is supported also by PRIN PNRR ``Linear and Nonlinear PDE's: New directions and Applications''.
The third author is supported by Princess Nourah bint Abdulrahman University Researchers Supporting
Project number (PNURSP-HC2023/3), Princess Nourah bint Abdulrahman University, Saudi Arabia}

\address[S.\ Mosconi]{Department of Mathematics and Computer Science
    \newline\indent
    University of Catania
    \newline\indent
    Viale A. Doria 6, I-95125 Catania, Italy}
\email{sunra.mosconi@unict.it}

\address[G.\ Riey]{Dipartimento di Matematica e Informatica,
    \newline\indent
    Universit\`a della Calabria
    \newline\indent
    Ponte Pietro Bucci 31B, I-87036 Arcavacata di Rende, Cosenza, Italy}
\email{giuseppe.riey@unical.it}

\address[M.\ Squassina]{College of Science
    \newline\indent
    Princess Nourah Bint Abdul Rahman University
    \newline\indent
    Saudi Arabia, Riyadh, PO Box 84428}
\email{marsquassina@pnu.edu.sa}

\address[M.\ Squassina]{Dipartimento di Matematica e Fisica
    \newline\indent
    Universit\`a Cattolica del Sacro Cuore
    \newline\indent
    Via della Garzetta 48, I-25133 Brescia, Italy}
\email{marco.squassina@unicatt.it}

\begin{document}

\begin{abstract}
We prove concavity properties for solutions to anisotropic quasi-linear equations,
extending previous results known in the Euclidean case.
We focus the attention on nonsmooth anisotropies and in particular
we also allow the functions describing the anisotropies to be not even.
\end{abstract}

\maketitle

\begin{center}
    \begin{minipage}{9cm}
        \small
        \tableofcontents
    \end{minipage}
\end{center}

\section{Introduction}\label{intro}

A natural question in the framework of nonlinear   elliptic PDEs is
whether a solution
inherits some qualitative properties from it domain of definition.
Starting from \cite{GNN79} 
extensive research has been developed in order to deduce symmetry of solutions %inherited
from the symmetry of the domain, via the so called Alexandroff-Serrin moving plane method.
But when the symmetry of the domain is dropped, one may wonder if the solutions still exhibit
some convexity properties just from the convexity of their domain.
As it turns out, classical concavity of, say, positive solutions with zero Dirichlet boundary conditions is often rather demanding:
it can be achieved for the torsion problem
\beq
\label{tor}
\begin{cases}
-\Delta u=1 &\text{in $\Omega$}\\
u=0&\text{on $\partial\Omega$},
\end{cases}
\eeq
 for example,
only for convex $\Omega$ which are suitably small perturbations of ellipsoids \cite{HNST18} (see also \cite{CW}) 
and the first eigenfunctions of the Dirichlet Laplacian are actually never concave, in any convex domain \cite[Remark 3.4]{Kaw85R}.
One may instead look for {\em quasi-concavity} of positive solutions in convex domains, meaning that all their  super-level sets are convex. This is usually accomplished by requiring that for a suitable strictly increasing functions $\varphi$
the composition $\varphi(u)$ is concave, a property called {\em $\varphi$-concavity} of $u$.
Indeed, in the seminal paper \cite{ML71}  
it is shown that the solution of the torsion problem \eqref{tor}  for $\Omega$ convex is such that
$\sqrt{u}$ is concave and in \cite{BL76}  
the authors show that the logarithm of a first positive eigenfunctions of the Dirichlet Laplacian is always concave if the domain is convex.
More generally, for the positive solution of
\[
\begin{cases}
-\Delta u=u^\beta &\text{in $\Omega$}\\
u=0&\text{on $\partial\Omega$},
\end{cases}
\]
in a convex $\Omega$, the power $u^{(1-\beta)/2}$ is concave for any $\beta\in [0,1]$,
the $0$-th power being formally identified with $\log u$, see \cite{Ke,Kor83Co}. These last two investigations, based on the so-called {\em  concavity function method}, gave birth to a rich research field on quasi-convexity properties of solutions to PDEs in the eighties, and we refer to \cite{Kaw85R} for the relevant bibliography. The concavity function method was also successfully applied to quasilinear equation of $p$-Laplacian type in \cite{Sa}. Later, a new approach to these problems, the {\em convex envelope method}, was introduced in \cite{ALL} in the framework of viscosity solutions to fully nonlinear PDEs.

\medskip

{\bf Recent contributions}.
These two strategies for investigating convexity properties of solutions to PDEs have been revisited, extended and modified in various ways, see for instance \cite{grecoporru, greco} for the concavity function method and \cite{CS, BS, INS} for the convex envelope method.

More recently, a general class of reactions $f$ ensuring quasi-concavity of the positive solutions of
\beq
\label{genf}
\begin{cases}
-\Delta u=f(u) &\text{in $\Omega$}\\
u=0&\text{on $\partial\Omega$},
\end{cases}
\eeq
 (and more generally of quasi-linear problems
involving the $p$-Laplace operator)  has been singled out in
\cite{BoMoSq1,BoMoSq2}, providing a precise connection on how the quasi-concavity of the solution is affected by the nonlinear term $f$ through the above mentioned $\varphi$. Indeed, the class of continuous increasing $\varphi$ can be partially ordered in a natural way according to their concavity, and one of the goals of \cite{BoMoSq1} was, given $f$ obeying suitable conditions, to determine a "minimally concave" $\varphi$ ensuring $\varphi$-concavity of the solutions of the corresponding quasilinear problem.
Note that, for general positive reactions $f$, quasi-concavity of positive solutions of \eqref{genf} can fail for some smooth convex $\Omega$, as shown in \cite{hns}.

In \cite{AAGS} the optimality of the assumptions of \cite{BoMoSq1} was discussed and
the results were then extended to cover positive solutions to
the quasi-linear problem
\[
-{\rm div}(\alpha(u) D u) + \frac{1}{2} \alpha'(u) |D u|^2 = f(u)
\]
(coupled with zero boundary conditions), related to the so called modified nonlinear Schrödinger equation, under suitable joint
hypothesis on $\alpha$ and $f$.

Another direction recently investigated in the literature is the quasi-concavity of the solutions
if the nonlinearity is perturbed \cite{BucurS} or if the equation is
nonautonomous in the diffusion or on the source term,
like in the second order semilinear problem
\[
-{\rm Tr}\, \left(A(x)\, D^2u\right)=a(x) \, u^\beta,
\]
 see \cite{ABCS}. Finally, {\em strict} quasi-concavity of positive solutions to \eqref{genf} has also been investigated, but turns out to be a much more delicate theme. We refer to \cite{BoMoSq2} for the relevant literature and open questions in the quasilinear case.

\medskip

{\bf Equations considered}. In this paper, given a bounded convex $\Omega\subseteq \R^N$, we investigate  the quasi-concavity of positive solutions of
\begin{equation}\label{eq iniziale}
\left\{
\begin{array}{ll}
-{\rm div}\, \left(\n H(\n u)\right)=f(u) & \hbox{ in } \Om\\
u=0 & \hbox{ on } \partial\Om\,,
\end{array}
\right.
\end{equation}
where
$H$ is a continuous convex $p$-homogenous function for some fixed $p>1$ vanishing only at the origin and $f$ is continuous and fulfils  suitable concavity conditions detailed below.
We are particularly interested in the case where $H$ is not even and possibly non-smooth, in which case \eqref{eq iniziale} requires a suitable variational formulation due to the possible lack of differentiability of $H$.

When $H$ is even, the operator appearing on the right of \eqref{eq iniziale}
is also known as the \emph{Finsler $p$-Laplacian}, since the corresponding kinetic energy
\[
u\mapsto \int_\Omega H(\n u)\, dx
\]
has density which can be expressed as
\[
H(\n u)=\f(\n u)^p
\]
for a suitable Finsler norm $\f$. Clearly, when the norm is the standard Euclidean one, we are reduced to the usual $p$-Laplacian.

This kind of energies can be used to model several anisotropic phenomena,
related for instance to continuum mechanics, image processing and biology \cite{AnInMa,CaHo,EsOs,Gu,PeMa}.
In materials science and chemistry a central role is played
by {\em non-smooth} Finsler norms in order to describe the behavior of crystalline microstructures
\cite{BeNoPa1,BeNoPa2,BeNoRi,Ta1,Ta2}.
Non-smooth Finsler norms are also used in control theory to describe the cost functional in some optimization problems \cite{Dav}.
Moreover, in differential geometry it is possible to consider
non-even energy densities $H$, as for instance those related to
Randers metrics \cite{BaChSh,ChSh}, which have applications also in relativity \cite{Ra}.

Problem \eqref{eq iniziale} and the qualitative properties of its solutions has been thoughtfully investigated in recent years under the assumption that $H$ is smooth and its corresponding Finsler norm has strongly convex unit ball, see for instance
\cite{BeFeKa,CaRiSc,CiFiRo,CoFaVa1,CoFaVa2}. Even under these more stringent assumptions, however, the quasi-concavity of solutions to \eqref{eq iniziale} was not generally  known.

Note that \eqref{eq iniziale} has to be understood, for crystalline non-differentiable energies, in weak sense either as a differential inclusion or as minimisation property of the corresponding energy. This will be made precise in the forthcoming paragraph. The regularity of solutions of \eqref{eq iniziale} is therefore very poor, at best $C^\alpha(\overline\Omega)$ for some $\alpha\in \ ]0, 1[$, and this poses serious issues in the direct applicability of both the aforementioned methods. It is worth noting that even when $H$ is smooth away from the origin, as for instance in the model case of the $p$-Laplacian whose corresponding solutions enjoy better regularity, the convex envelope technique of \cite{ALL} still runs into problems. The very notion of viscosity solution is quite different from the standard one, and  the coincidence of weak and viscosity solutions for general quasilinear degenerate/singular equations is object of contemporary research, mainly built around the ideas of \cite{JJ}. We are not aware of any result of this kind for the general crystalline case we are considering in the present investigation. 
\medskip

{\bf Main result}. \  In order to state our main result, given a continuous $f\in C^0(\R_+)$, set
$$
F(t)=\int_0^t f(s)\, ds
$$
and  let $\vf$ be
\begin{equation}\label{def vf}
\vf(t)=\int_1^tF(s)^{-\frac 1 p}ds\,.
\end{equation}
which is well defined on $]0, M_f]\cap \R$, for
\beq
\label{defM}
M_f=\inf\{t>0: f(t)\le 0\}
\eeq
(in most instances we will actually have $M_f=+\infty$).
Note that in general $\varphi$ may be unbounded near $0$ and also possibly near some  $\bar t>M_f$. In order to deal with possibly non-smooth convex $H$, we restrict to special variational solutions of \eqref{eq iniziale}, namely minimisers of the corresponding energy
\beq
\label{energy}
\left\{w\in W^{1,p}_0(\Omega): w\ge 0\right\}\ni w\mapsto J(w)=\int_\Omega \frac{1}{p}\, H(Dw)-F(w)\, dx.
\eeq
This is not restrictive, since if $H$ is differentiable and $f$ fulfils the additional requirements specified in the statement below,  any positive solution of \eqref{eq iniziale} turns out to minimise \eqref{energy}.
\begin{theorem}\label{teorema 1}

Suppose that $H\in C^0(\R^N)$ is convex, positively $p$-homogeneous and vanishes only at the origin, while $f\in C^\alpha(]0, M_f]), \R)\cap C^0([0, \infty[, \R)$ fulfils  $M_f>0$.
Let $u\in W^{1,p}_0(\Om)$ be a non-negative, nontrivial minimiser for \eqref{energy}.

Assume that
\begin{enumerate}
\item $F^{\frac 1 p}$ is concave and  $F/f$ is convex on $]0,M_f[$
\end{enumerate}
and  one of the following conditions
\begin{enumerate}
\item[($2_H$)]
$H$ is strictly convex
\item[($2_F$)]
$t\mapsto F(t^{1/p})$ is strictly concave.
\end{enumerate}
Then $u\le M_f$ and the function $v=\vf(u)$ is concave in $\Omega$, where $\vf$ is given in \eqref{def vf}.
\end{theorem}

{\bf Comments on the statement}.
\begin{itemize}
\item
We are not assuming any regularity or evenness hypothesis on $H$ in Theorem \ref{teorema 1}. This lack of regularity and symmetry on $H$ forces us to build suitable tools  such as comparison principles and Hopf type Lemma without relying on PDE arguments and are therefore, as far as we know, new.
\item
Due to the assumed $p$-positive homogeneity of $H$, its strict convexity is equivalent to the strict convexity of $\{H\le 1\}$. Note that in general the {\em strong} convexity of $\{H\le 1\}$ (meaning that the principal curvature of its boundary are positively bounded from below) is required to get classical $C^{1,\alpha}$ regularity of the corresponding solution.
\item
By elementary means, one can show that if $t\mapsto F^{1/p}(t)$ is concave, so is $t\mapsto F(t^{1/p})$, which in turn means that $t\mapsto f(t)/t^{p-1}$ is non-increasing, a common notion in Brezis-Oswald uniqueness type results, cfr. \cite{BO, DiSa}.
\item
 If $H\in C^1(\R^N)$ and $t\mapsto F(t^{1/p})$ is concave it will turn out that any solution of \eqref{eq iniziale} is actually a non-negative minimiser for $J$. This fact can be proved (see \eqref{varchar}) in the more general framework of possibly non-differentiable $H$, by using the notion of {\em energy critical point} developed in Section \ref{SECP} and inspired by \cite{DGM}.
\item
Either of condition {\rm ($2_H$)} and {\rm ($2_F$)} above, coupled with the concavity of $t\mapsto F(t^{1/p})$ implies that the non-negative minimisers of $J$ are essentially unique. More precisely, either $t\mapsto F(t^{1/p})$ is linear and  $u$ is a first Dirichlet positive eigenfunction minimising the corresponding Rayleigh quotient (see \eqref{RQi} below), or $u$ is the unique non-negative minimiser of $J$ on $W^{1,p}_0(\Omega)$. This has been proved in \cite{M}, see Proposition \ref{prouniqueness} for a precise statement.
\item
Existence of non-negative minimisers (or, equivalently, of non-negative critical points) for $J$ can be characterised in terms of the asymptotic behaviour at $0$ and $+\infty$ of $t\mapsto f(t)/t^{p-1}$, assuming it is non-increasing. This is the content of a Brezis-Oswald type result proved in Proposition \ref{proBO} below, which complement the results of \cite{M}.
\item
Assumption ($1$) can be weakened to
\begin{enumerate}
\item[(1')]
\ {\em $t\mapsto f(t)/t^{p-1}$ is non increasing and $t\mapsto e^{(p-1) t}/f(e^t)$ is convex}
\end{enumerate}
allowing to establish $\log$-concavity of $u$. This follows coupling the arguments in \cite{BoMoSq1} with tools developed in this manuscript. Note that   ($1$) implies ($1'$), but in this case if $\varphi$-concavity is strictly stronger that $\log$-concavity.
\item
The assumption $f\in C^\alpha(]0, M_f])$ is a technical one. Indeed, from the convexity of $F/f$ and the positivity of $f$ in $]0, M_f[$ one infers that $f\in {\rm Lip}_{\rm loc}(]0, M_f[)$ so that we actually require a H\"older control at $M_f$ alone.

\end{itemize}

 \medskip

{\bf Applications}.

\begin{itemize}
\item
A natural choice for the energy density is $H(z)=|z|_r^p$ for given $r\in [1, \infty]$ and $p\in \ ]1, \infty[$, where $|z|_r$ denotes the $\ell^r$ norm on $\R^N$.  Note that with this choice $\{H\le 1\}$  is never strongly convex unless $r=2$, but $H$ is strictly convex for any $r\in \ ]1, \infty[$, so that assumption ($1$) on the reaction suffices to get quasi-concavity of the corresponding solutions of \eqref{eq iniziale}. If $r=1$ or $\infty$, $H$ fails to be strictly convex and assumption ($2_F$) is needed to obtain quasi-concavity of the minimisers.\\
In the non-even setting we can choose any convex, bounded, open $K\subseteq \R^N$ containing $0$ (but not necessarily symmetric) and consider its Minkowski functional
\[
\Phi(z)=\inf\{t>0:z/t\in K\}.
\]
defining the energy density  as $H=\Phi^p$. The resulting $H$ fulfils ($2_H$) as long as $K$ is strictly convex.
\item
Given $p>1$, the typically used reaction  is $F(t)=c\, t^q$ for $1\le q<p$, $c>0$, which then fulfils (1) and ($2_F$) above. In this case, given any convex, positively $p$-homogeneous $H$ vanishing only at the origin and a convex bounded $\Omega\subseteq \R^N$, a non-negative minimiser  $u$ of $J$ has the property that $u^{(p-q)/q}$ is concave.
In this case, since ($2_F$) holds true,  $H$ may fail to be strictly convex allowing for example to establish the power-concavity of the minimisers of the non-homogeneous Rayleigh quotient
\[
\inf\left\{\dfrac{\displaystyle{\int_\Omega |Du|_\infty^2\, dx}}{\displaystyle{\left(\int_\Omega |u|^q\, dx\right)^{2/q}}}: u\in W^{1, 2}_0(\Omega)\setminus\{0\}\right\}
\]
for any $q\in [1,2[$ and convex, bounded $\Omega$.\\
Other explicit examples where ($2_F$) holds true, thus allowing such a generality for $H$, can be found in \cite[Section 2]{BoMoSq1}.
\item
Another application of our main result is when $u$ is a first Dirichlet positive eigenfunction of a convex bounded domain $\Omega$, thus minimising the homogeneous Rayleigh quotient
\beq
\label{RQi}
\lambda_{1, H}^+(\Omega)=\inf\left\{\dfrac{\displaystyle{\int_\Omega H(Du)\, dx}}{\displaystyle{\int_\Omega u^p}}: u\in W^{1,p}_0(\Omega)\setminus\{0\}, u\ge 0\right\}.
\eeq
Then $F(t)=\lambda_{1, H}^+(\Omega) \, t^p/p$ and $u$ is $\log$-concave as long as $H$ is strictly convex. Note that $F$ fulfils (1) in this case, but not ($2_F$), and that the requirement $u\ge 0$ in \eqref{RQi} is needed since $H$ may fail to be even.

A sample consequence is the $\log$-concavity of the first positive Dirichlet eigenfunction of the {\em pseudo} $p$-Laplacian studied in \cite{BK}, solving
\[
-\tilde{\Delta}_p u=-\sum_{i=1}^N |\partial_i u|^{p-2}\partial_i u=\tilde{\lambda}_{1, p}(\Omega)\, u^{p-1}
\]
for $p\in\  ]1, \infty[$. The energy-density of the corresponding kinetic energy is $H(z)=|z|_p^p$ which, as already noted, is strictly convex but $\{H\le 1\}$ is not strongly convex.
\end{itemize}

\medskip

 {\bf Sketch of proof}.\
 The proof of Theorem \ref{teorema 1} relies on a two-steps approximation tailored on both $H$ and $F$. First we smooth out $H$, by keeping its positive $p$-homogeneity and ensuring a form of strong $p$-ellipticity that the original $H$ may not have. The uniqueness (up to scalar multiples when $t\mapsto F(t^{1/p})$ is linear) of the minimiser proved in Section \ref{SECP} is key to grant the convergence of  the minimisers corresponding to the smoothed functional to the original one. Thus we are reduced to prove $\varphi$-concavity of a minimiser $u$ when $H$ is smooth and $p$-elliptic. Standard regularity theory ensures in this setting up to $C^{1, \alpha}(\overline{\Omega})$ regularity, but since $H$ is not assumed to be even, nor can be its regularisations. An appropriate and apparently new anisotropic version of the Hopf Lemma is proved in  subsection \ref{SHL} and turns out to be essential for the second regularisation procedure, since we can infer from it uniform $C^2$ bounds in an inner thin strip arbitrarily near the boundary of $\Omega$. A family of different approximating problem is then built, whose corresponding minimisers are {\em globally} $C^2$. The form of this approximation (see \eqref{afunc}) has be to chosen carefully, in order to ensure that classical results of Kennington and Korevaar can be applied under conditions involving solely the reaction $f$.
 Then the strategy of Sakaguchi \cite{Sa} can be employed, namely to consider separately the concavity function related the corresponding solutions far from $\partial\Omega$ and on the boundary of the aforementioned strip, where uniform $C^2$ bounds are available. By passing to the limit, this allows to conclude the concavity of $\varphi(u)$ on any strongly convex subdomain sufficiently close to $\Omega$, and thus the theorem.

 As a final point of interest it may be worth mentioning some cases which, despite natural, are not covered by Theorem \ref{teorema 1}. One may consider the crystalline, $2$-homogeneous energy $H(z)=|z|_\infty^2$
 and given a convex $\Omega\subseteq \R^2$, look for a positive minimiser $u\in W^{1,2}_0(\Omega)$ of the Rayleigh quotient \eqref{RQi}. Note that $H$ is not strictly convex and $F(t^{1/2})=\lambda_{1, H}(\Omega)\, t/2$ is not strictly concave, thus we are not able to prove through Theorem \ref{teorema 1} that $\log u$ is concave, as one may naively guess. The reason, as should be clear from the previously described proof of Theorem \ref{teorema 1}, is that we don't known wether the corresponding eigenvalue is simple, a fact that may well be false for some convex $\Omega$.

\medskip

{\bf Notations}.\ In the paper $c$ and $C$ (eventually with subscripts) denote constants
which are allowed to vary from line to line; their dependance on various parameters will be outlined only when relevant to the proof. For $t\in \R$ we denote $t_+=\max\{t, 0\}$ and $t_-=\max\{-t, 0\}$.

For $a,b\in\R^N$ we denote by $(a, b)$ the standard Euclidean scalar product, by $|a|$ the Euclidean norm and by $a\otimes b$ the matrix whose entries
are $(a\otimes b)_{ij}=a_i b_j$. Recall that, for $v,w\in\R^N$,
there holds:
\[
\big( a\otimes b\, v,w\big)=\big(b,v\big)\,  \big( a,w\big).
\]
For a measurable $E\subseteq\R^N$, we let $|E|$ be its $N$-dimensional Lebesgue measure and for $p\ge 1$, the $L^p(E)$ norm of a measurable $u:E\to \R$ will be denoted by $\|u\|_p$ when omitting the domain $E$ of $u$ causes no confusion.
\section{Preliminary results}

\subsection{Main assumptions}
Throughout the paper $\Omega$ will be an open subset of $\R^N$ with finite measure, often assumed to be convex and bounded. Recall that a strongly convex set is a smooth convex set such that the principal curvatures of $\partial\Omega$ are positive. Clearly, any strongly convex $\Omega$ is strictly convex, but the opposite may not be true.

Moreover, $H:\R^N\to [0, \infty[$ will denote a continuous convex function, obeying at least the one-sided bound
\beq
\label{pcontrol2}
H(z)\ge \frac{1}{C}\, |z|^p.
\eeq
A strengthening of the previous condition will be often assumed, namely
\beq
\label{pcontrol}
\frac{1}{C}\, |z|^p\le H(z)\le C\, |z|^p
\eeq
and in many instances $H$ will be additionally required to be positively $p$-homogeneous ($p>1$), meaning
\[
H(t\, z)=t^p\, H(z)\qquad \forall t>0, z\in \R^N.
\]
 Any such $H$ clearly obeys \eqref{pcontrol}.

The reaction $f$  belongs to $C^0(\R)$, is even and satisfies the one-sided growth condition
\beq
\label{growthf}
f(t)\le C\, \big(t^{p-1}+1\big)\qquad t\ge 0
\eeq
as well as $M_f>0$, where $M_f$ is given in \eqref{defM}, (possibly $M_f=+\infty$). Let us remark that the evenness condition on $f$ is assumed only for convenience, since we are interested in non-negative critical points for the corresponding functional.
Given such  a function, we will  set
\[
F(t)=\int_0^t f(s)\, ds.
\]
and
\[
f_+(t)=\max\{0, f(t)\}, \qquad F_+(t)=\int_0^t f_+(s)\, ds.
\]
Note that we are making a slight abuse of notation here as $F_+(t)\ne (F(t))_+$.
We will often assume (see e.\,g.\,the following paragraph) that $F^{1/p}$ is concave  on $[0, M[$. Note that from the concavity of $F^{1/p}$ on $[0, \infty[$ we readily infer \eqref{growthf} and, more importantly the following condition
\beq
\label{hypf}
\R_+\ni t\mapsto \frac{f(t)}{t^{p-1}}\quad \text{is non-increasing},
\eeq
which in turn is equivalent to the concavity of $t\mapsto F(t^{1/p})$ on $\R_+$.
Note that the opposite implication is not true, i.\,e.\,$t\mapsto F(t^{1/p})$ may be concave but $t\mapsto F^{1/p}(t)$ may fail to be concave, see \cite[Remark 3.4]{BoMoSq1}.

\subsection{Concavity function}
Given a continuous function $v:\Omega\to \R$ with $\Omega$ convex, its {\em convexity function} $c:\Omega\times\Omega\times[0, 1] \to \R$ is defined as
\[
c_v(x, y, t)= t\, v(x)+ (1-t)\, v(y)-v\big(t\, x+(1-t)\, y\big).
\]
Clearly, $v$ is concave in $\Omega$ if and only if $c_v\le 0$ in its domain.  We recall the following fundamental properties of the concavity function and its relation with solutions of PDE. Recall that a function $g:G\subseteq \R^m\to \R$ with $G$ convex is called {\em harmonic concave} if for any $x, y\in G$ such that $g(x)+g(y)>0$ it holds
\[
\big(g(x)+g(y)\big) \, g \big(\frac{x+y}{2}\big)\ge 2\, g(x)\, g(y)
\]
If $g$ is positive, this is equivalent to the convexity of $1/g$.

\begin{proposition}[\cite{grecoporru}]\label{propKennington}
Let $\Omega$ be bounded and convex in $\R^N$, $N\ge 2$ and $v\in C^2(\Omega)$ solve
\[
-{\rm Tr}\,\left( A(Dv)\, D^2 v\right)= b( x, v, Dv)
\]
where $A\in C^1(\R^N, \R^{N\times N})$ fulfills   for some  $0<\lambda\le \Lambda<\infty$
\[
\lambda\, |\xi |^2\le \left(A(z)\, \xi, \xi\right)\le \Lambda\, |\xi|^2\qquad \text{ for all $z, \xi\in\R^N$}
\]
while  $(x, t, p)\mapsto b(x, t, p)$ is continuous, differentiable with respect to $x$ and $p$ and
 \[
 \partial_x b, \partial_p b\in L^\infty_{\rm loc}(\Omega\times \R\times \R^N).
 \]
 If
\begin{enumerate}
\item
$t\mapsto b(x, t, p)$ is non-increasing on $v(\Omega)$ for any $(x, p)\in\Omega\times Dv(\Omega)$
\item
$ (x, t)\mapsto b(x, t, p)$ is harmonic concave on $\Omega\times v(\Omega)$ for any $p\in \R^N$
\end{enumerate}
then $c_v$ cannot  attain a positive interior maximum in $\Omega\times \Omega\times[0, 1]$.
\end{proposition}

\begin{proof}
This is a rephrasement of \cite[Theorem 2.1]{grecoporru} applied to  $\hat{v}=-v$ with $\hat{A}(p)= A(-p)$, $\hat{b}(x, t, p)= b(x, -t,  -p)$. The proof in \cite{grecoporru} uses the $C^1$ regularity of both $\hat{A}$ and $\hat{b}$ to prove that the {\em convexity} function
\begin{equation}
\label{2conc}
\hat{v}\big(\frac{x+y}{2}\big)- \frac{\hat{v}(x)+\hat{v}(y)}{2}
\end{equation}
satisfies a differential inequality ensuring that it cannot attain a positive interior maximum in $\Omega\times\Omega$. The same proof shows that {\em for any given}  $t\in [0, 1]$, the function
\[
(x, y)\mapsto \hat{v}(t\, x+(1-t)\, y)-t\, \hat{v}(x)-(1-t)\, \hat{v}(y)
\]
cannot attain a positive maximum in $\Omega\times\Omega$. This stronger statement  {\em a fortiori} implies that
\[
(x, y, t)\mapsto \hat{v}(t\, x+(1-t)\, y)-t\, \hat{v}(x)-(1-t)\, \hat{v}(y)
\]
cannot attain a positive interior maximum in $\Omega\times\Omega\times [0, 1]$. We will show how to remove the regularity assumption on $\hat{b}$ in the proof of \cite{grecoporru} for the the convexity function \eqref{2conc}.
To this end, note that it suffices to study the convexity function  near points $(x, y)\in \Omega\times\Omega$ such that
\beq
\label{posconv}
\hat{v}\big(\frac{x+y}{2}\big)- \frac{\hat{v}(x)+\hat{v}(y)}{2}>0,
\eeq
which form an open set by continuity.
The only point where the regularity of $t\mapsto \hat{b}(x, t, p)$ (lacking in our setting) is used is applying  Lagrange theorem to deduce the inequality
\[
\hat{b}\big(z, \hat{v}(z), D\hat{v}(z) \big)\ge \hat{b}\Big(z, \frac{\hat{v}(x)+\hat{v}(y)}{2}, D\hat{v}(z)\Big)+d(x, y) \, \Big( \hat{v}(z)- \frac{\hat{v}(x)+\hat{v}(y)}{2}\Big)
\]
for suitable $d(x, y)\ge 0$, where  $z=(x+y)/2$. However, this inequality  is only needed  at points $(x, y)\in \Omega\times\Omega$ fulfilling \eqref{posconv},
in which case it is certainly true with $d=0$ regardless of the regularity of $\hat{b}(z, \cdot, p)$, since $\hat{b}$ is non-decreasing.

\end{proof}

In the following, given $\Omega\subseteq \R^N$ and $\delta>0$, we set
\begin{equation}
\label{defsdelta}
 \Omega_\delta=\{x\in \Omega: \delta< {\rm dist}(x, \partial\Omega) \}.
 \end{equation}

\begin{proposition}[\cite{Kor83Co}, Lemma  2.4 ]\label{propKorevaar}
Suppose that $\Omega$ is smooth, bounded and strongly convex and $u\in C^1(\overline{\Omega})\cap C^2(\overline{\Omega}\setminus \Omega_\eta)$ for some $\eta>0$ is  such that
\beq
\label{assu}
  u>0\quad \text{in $\Omega$}, \qquad u=0\quad \text{on $\partial\Omega$}, \qquad \frac{\partial u}{\partial n}>0 \quad \text{on $\partial\Omega$}.
  \eeq
  If $\varphi\in C^2(\R_+; \R) $ fulfils
 \beq
 \label{assKor}
\lim_{t\to 0^+}\varphi'(t)=+\infty  , \qquad \varphi''<0<\varphi' \ \text{near $0$},\qquad  \lim_{t\to 0^+}\frac{\varphi(t)}{\varphi'(t)}= \lim_{t\to 0^+}\frac{\varphi'(t)}{\varphi''(t)}=0,
 \eeq
 set $v=\varphi(u)$. Then  there exists $\delta\in \ ]0, \eta[$ such that
 \beq
 \label{derseconda}
 D^2v<0 \quad \text{on } \Omega\setminus\Omega_\delta
 \eeq
and for all $x_0\in\Omega\setminus\Omega_\delta$ and $x\in\Omega\setminus\{x_0\}$ it holds
\beq
\label{sottopiani}
v(x_0)+(Dv(x_0), x-x_0)> v(x).
\eeq
\end{proposition}

\begin{proof}
The proof of \eqref{derseconda} is in \cite[Lemma 2.4, fact 2]{Kor83Co}. Let then $\delta_0>0$ be such that $D^2v<0$ in $\Omega\setminus\Omega_{\delta_0}$.
We give another   proof of \eqref{sottopiani} since the last part of \cite[Lemma 2.4]{Kor83Co} is a bit obscure.
Let
\[
A_{x_0}=\left\{x\in \Omega:v(x_0)+(Dv(x_0), x-x_0)\le v(x)\right\}.
\]
Fix $\delta_1=\delta_1(\delta_0, \Omega)\in \ ]0, \delta_0[$ such that
\beq
\label{claimG}
x_0\in \Omega,\quad {\rm dist}\, (x_0, \partial\Omega)<\delta_1\quad \Longrightarrow \quad B_{\delta_1}(x_0)\cap \Omega\subseteq \Omega\setminus\Omega_{\delta_0}.
\eeq
We then claim that for sufficiently small $\delta<\delta_1$ (depending on $u$ as well) it holds
\beq
\label{claimK}
x_0\in \Omega,\quad{\rm dist}\, (x_0, \partial\Omega)<\delta\quad \Longrightarrow\quad  A_{x_0}\subseteq B_{\delta_1}(x_0)\cap \Omega.
\eeq
Before proving the claim, let us note that it implies \eqref{sottopiani} thanks to \eqref{claimG} and the strict concavity of $v$ in the convex open set $B_{\delta_1}(x_0)\cap \Omega$ granted by \eqref{derseconda}, which force $A_{x_0}=\{x_0\}$.

We prove \eqref{claimK} by contradiction, thus assuming that there exists $\delta_n\downarrow 0$, $x_n\in \Omega\setminus\Omega_{\delta_n}$ and $y_n\in \Omega$ such that
\[
|y_n-x_n|\ge \delta_1, \qquad v(x_n)+(Dv(x_n), y_n-x_n)\le  v(y_n).
\]
By compactness we can suppose $x_n\to \bar x\in \partial\Omega$ and $y_n\to \bar y\in \overline{\Omega}$, with $|\bar y-\bar x|\ge \delta_1$. Denoting by $n(\bar x)$ the interior normal to $\partial \Omega$ at $\bar x$, it holds
\[
\lim_n(Du(x_n), y_n-x_n)= (Du(\bar x), \bar y-\bar x)=|Du(\bar x)|\, (n(\bar x),  \bar y-\bar x)>0
\]
by the strict convexity of $\partial\Omega$ and \eqref{assu}. Therefore there exists $\theta>0$ such that for sufficiently large $n$ it holds
\[
(Du(x_n), y_n-x_n)>\theta.
\]
Recalling the definition of $v$, we thus have
\[
\varphi'(u(x_n))\left(\frac{\varphi(u(x_n))}{\varphi'(u(x_n))}+(Du(x_n), y_n-x_n)\right)= v(x_n)+(Dv(x_n), y_n-x_n)\le  v(y_n)=\varphi(u(y_n))
\]
so that for sufficiently large $n$
\[
\varphi'(u(x_n))\left(\frac{\varphi(u(x_n))}{\varphi'(u(x_n))}+\theta\right)\le \varphi(u(y_n)).
\]
However, since $u(x_n)\to 0$, the left hand side goes to $+\infty$ by \eqref{assKor} while the right and side is bounded from above (since $\varphi$ is smooth on $]0, +\infty[$ and increasing near $0$). This proves claim \eqref{claimK} and then \eqref{sottopiani}.
\end{proof}

The next proposition shows the r\^ole of condition \eqref{sottopiani} in analysing the boundary behaviour of the convexity function. Notice that we will apply it for convex domains slightly smaller than the domain of definition of the function $v$ defined above, in order to ensure that it is smooth up to the boundary.
\begin{proposition}[\cite{Kor83Co}, Lemma  2.1 ]\label{propKorevaar2}
Suppose that $\Omega$ is smooth, bounded and strongly convex and $\eta>0$. If $v\in C^1(\overline{\Omega})$ fulfils \eqref{sottopiani} for all $x_0\in \partial\Omega$ and $x\in \overline{\Omega}$, then  $c_{v}$  cannot attain a positive maximum on $\partial(\Omega\times\Omega)\times [0, 1]$.
\end{proposition}

Clearly, condition \eqref{sottopiani} is not stable under $C^2$ convergence, but the conjunction of \eqref{derseconda} and \eqref{sottopiani} is, as has been observed in \cite{BoMoSq1}. We report the argument therein for sake of completeness.

\begin{proposition}\label{prostable}
Let $\Omega$ be smooth, bounded and strongly convex, $\eta>0$ and let  $v_n\in C^1(\overline{\Omega})\cap C^2(\overline{\Omega}\setminus \Omega_\eta)$ be such that $v_n\to v$ in $C^1(\overline{\Omega})$ and in $C^2(\overline{\Omega}\setminus \Omega_\eta)$. If  $v$ fulfils \eqref{derseconda} for some $\delta\in \ ]0, \eta]$ and \eqref{sottopiani} at all points $x_0\in \partial\Omega$ and $x\in \overline{\Omega}\setminus\{x_0\}$ then, for all sufficiently large $n$, $v_n$ fulfils them as well.
\end{proposition}

\begin{proof}
Clearly \eqref{derseconda} holds true for sufficiently large $n$, which we'll assume henceforth. If \eqref{sottopiani} does not hold, for a (not relabelled) subsequence there are points $x_n\in \partial\Omega$ and $y_n\in \overline{\Omega}$ with $y_n\neq x_n$ and
\beq
\label{bmscontr}
v_n(x_n)+(Dv_n(x_n), y_n-x_n)\le v_n(y_n).
\eeq
By the smoothness and strong convexity of $\partial\Omega$ we can find $c>0$ such that $B_{c\delta}(\bar x)\cap \Omega\subseteq \Omega\setminus\Omega_\delta$. Let $n$ be so large that $B_{c\delta/2}( x_n)\subseteq B_{c\delta}(\bar x)$. Since $v_n$ is strictly concave on the convex open set $B_{c\delta/2}( x_n)\cap \Omega$, the point $y_n$ cannot lie in $B_{c\delta/2}(x_n)$, and thus
\[
|x_n-y_n|\ge c\, \delta/2.
\]
By taking a subsequence  we can suppose that  $x_n\to \bar x_0\in \partial\Omega$, $y_n\to\bar x\in  \overline{\Omega}$ and from the previous display we get  $|\bar x_0-\bar x|>0$. Taking the limit in \eqref{bmscontr}, we reach
\[
v(\bar x)_0+(Dv(\bar x_0), \bar x-\bar x_0)\le v(\bar x), \qquad \overline{\Omega}\ni \bar x\ne \bar x_0\in \partial\Omega,
\]
contradicting  assumption \eqref{sottopiani} for $v$.

\end{proof}

 The previous Propositions will eventually be applied to $v=\varphi(u)$, with $\varphi$ defined as in \eqref{def vf}, in a smaller domain $\Omega'\Subset \Omega$. To this end, we recall the following elementary facts  from \cite{BoMoSq1}.

 \begin{lemma}\label{lemmavarphi}
Let $f\in C^0([0, +\infty[, \R)$ fulfill \eqref{growthf}.
If $F^{1/p}$ is concave and $F/f$ is convex  on $]0, M_f]\cap \R$, with $M_f$ as in \eqref{defM}, then
\begin{enumerate}
\item
The function $\varphi$ defined in \eqref{def vf} is invertible and fulfils \eqref{assKor} on $]0, M_f[$.
\item
If $\psi=\varphi^{-1}$, then $\psi''/\psi'$ is non-increasing and $\psi'/\psi''$ is convex on $]0, \varphi(M_f)[$
\end{enumerate}
\end{lemma}

\section{Critical point theory}\label{SECP}
In this section we give a meaning to problem \eqref{eq iniziale}, which at the moment is oddly defined as $H$ may fail to be differentiable, by using the notion of {\em energy critical point}. We then study the existence and uniqueness of the corresponding energy critical points.

\subsection{Energy critical points}

Given a convex $H\in C^0(\R^N)$ fulfilling \eqref{pcontrol2}
 and an even $f\in C^0(\R)$ obeying \eqref{growthf}, the corresponding functional $J$  will be defined as
\[
J(v)=\int_\Omega \frac{1}{p}\, H(Dv)-F(v)\, dx, \qquad v\in W^{1,p}_0(\Omega).
\]
Note that $J$ is always well defined and $J(u)>-\infty$ for all $u\in W^{1,p}_0(\Omega)$, as $(F(v))_+\in L^1(\Omega)$ for any $v\in W^{1,p}_0(\Omega)$ (thanks to \eqref{growthf}, the finite measure assumption on $\Omega$ and Poincar\'e inequality) but,  as \eqref{pcontrol2} and \eqref{growthf} are only one-sided, the resulting   $J$ may assume the value $+\infty$. Note that under assumption \eqref{growthf}, $J$ is anyway sequentially l.\,s.\,c.\,in the weak topology of $W^{1,p}_0(\Omega)$, as
\[
-\int_\Omega F(v)\, dx=\int_\Omega (F(v))_-\, dx -\int_\Omega (F(v))_+\, dx
\]
 and the first term is l.\,s.\,c.\,by Fatou lemma while the second is continuous by \eqref{growthf} and the finite measure assumption on $\Omega$.
As we are interested in non-negative solutions of \eqref{eq iniziale}, we let
\[
\big(W^{1,p}_0(\Omega)\big)_+:=\left\{v\in W^{1,p}_0(\Omega):v\ge 0\right\}.
\]
and define $u\in \big(W^{1,p}_0(\Omega)\big)_+$ to be a non-negative {\em  energy critical point} for $J$ on $\big(W^{1,p}_0(\Omega)\big)_+$ if $u$ minimises the corresponding semi-linearized convex functional
\beq
\label{semil}
v\mapsto J_u(v):=\int_{\Omega}\frac{1}{p}\, H(Dv)-f(u)\, v\, dx
 \eeq
 on $\big(W^{1,p}_0(\Omega)\big)_+$.
 Alternatively, an energy critical point $u$ for $J$ on the full $W^{1,p}_0(\Omega)$ is a minimiser of \eqref{semil} over
 \[
 V_u=\left\{v\in W^{1,p}_0(\Omega): \big(f(u)\, v\big)_+\in L^1(\Omega)\right\}.
 \]

These different definitions of critical point, taken from \cite{DGM}, deserve some comment. Note that $J_u$, as defined in \eqref{semil} is always well defined on $\big(W^{1,p}_0(\Omega)\big)_+$ under assumption \eqref{growthf}, thanks to the same argument as before and the condition $v\ge 0$. This is not the case for $J_u$ on the whole $W^{1,p}_0(\Omega)$ and in this case the requirement $\big(f(u)\, v\big)_+\in L^1(\Omega)$ is needed.

Note, moreover,    that as before \eqref{growthf} ensures that  $J_u(u)>-\infty$ while trivially $J_u(0)=0$, so that
\beq
\label{elle1}
\text{$u\in C_J$ or $C_J^+$}\qquad \Longrightarrow\qquad H(Du)\in L^1(\Omega), \quad f(u)\, u\in L^1(\Omega).
\eeq

The set of all energy critical points for $J$ on $\big(W^{1,p}_0(\Omega)\big)_+$ and $W^{1,p}_0(\Omega)$, respectively, will be denoted by $C_J^+$ and $C_J$, respectively.
By using the nonlinearity $f_+$ instead of $f$, we can define the   functional
\[
J_+(v)=\int_\Omega \frac{1}{p}\, H(Dv)-F_+(v)\, dx
\]
and the consider corresponding critical points $C_{J_+}$ and $C_{J_+}^+$.

Under quite general assumptions on $f$, all the previuos notions of critical points coincide.

\begin{lemma}\label{lemmaCj}
Let $\Omega\subseteq \R^N$ have finite measure, $H$ be convex and obey \eqref{pcontrol2} and $f\in C^0(\R)$ be even, fulfil \eqref{growthf}, and
\beq
\label{propfM}
f(t)>0 \quad \text{for $|t|<M_f$}, \quad f(t)\le 0\quad \text{for $|t|\ge M_f$}
\eeq
for some $M_f\ge 0$, with possibly $M_f=+\infty$.
 Then
\beq
\label{proCj+}
C_J =C_J^+=C_{J_+}^+=C_{J_+}
\eeq
and any energy critical points fulfils $0\le u\le M_f$
\end{lemma}

\begin{proof}
Denote by $C$ any one of the set $C_J, C_{J_+}, C_J^+$ or $C_{J_+}^+$.
We can assume that $f$ does not vanishes identically, otherwise the claim is trivial as all the previous sets are $\{0\}$. In this case, $M_f>0$ in \eqref{propfM}.
We claim that
\beq
\label{claim1}
u\in C\quad \Longrightarrow\quad |u|\le M_f.
\eeq
Recalling \eqref{elle1}, we test the minimality of $u$ with $v=\min\{u, M_f\}$ to get
\beq
\label{slitt}
\begin{split}
 \int_{\Omega}\frac{1}{p}\, H(Du)-f(u)\, u\, dx&\le  \int_{\Omega}\frac{1}{p}\, H(Dv)-f(u)\, v\, dx\\
 &= \int_{\{u\le M_f\}}\frac{1}{p}\, H(Du)-f(u)\, u\, dx - \int_{\{u>M_f\}}f(u)\, M_f\, dx.
 \end{split}
 \eeq
 Note from \eqref{propfM} and \eqref{elle1} that it holds
 \[
0\le  -\int_{\{u>M_f\}} f(u)\, M_f\, dx\le -\int_{\{u>M_f\}} f(u)\, u\, dx
\]
so that all terms in \eqref{slitt} are finite and rearranging we get
 \[
 \int_{\{u>M_f\}}\frac{1}{p}\, H(Du)-f(u)\, (u-M_f)\, dx\le 0
  \]
  proving that $u\le M_f$ thanks to \eqref{pcontrol2} and \eqref{propfM}. The same conclusion holds if $f$ is replaced by $f_+$, hence \eqref{claim1} is proved if $u\in \big(W^{1,p}_0(\Omega)\big)_+$, so it suffices to consider the case $u\in C=C_J$ or $C_{J_+}$. But in this case, $\hat u(x)=-u(-x)\in C$, where $\Omega$ is replaced by $-\Omega$, so that the previous argument ensures  $\hat u\le M_f$, i.\,e.\,$u\ge -M_f$, which concludes the proof of \eqref{claim1} in all cases.
Therefore $f(u)=f_+(u)\ge 0$ for any $u\in C$, hence the functionals in \eqref{semil} are the same for $f$ or $f_+$. It follows that
    \[
  C_{J_+}=C_J, \qquad C_{J_+}^+=C_J^+.
  \]
 Moreover, by \cite[Corollary 3.6]{M}\footnote{Note that only \eqref{pcontrol2} is used in the proof of point 1 therein.}, any  $u\in C_{J_+}$ is non-negative in $\Omega$, so that  $C_{J_+}\subseteq \big(W^{1,p}_0(\Omega)\big)_+$.
 Finally, again from $f(u)\ge 0$, we get that minimisation over $\big(W^{1,p}_0(\Omega)\big)_+$ of the functional in \eqref{semil} is equivalent to minimisation over the whole $W^{1,p}_0(\Omega)$, since
  \[
  \int_\Omega \frac{1}{p}\, H(Dv)-f(u)\,v\,  dx\ge \int_\Omega \frac{1}{p}\, H(Dv_+)-f(u)\,v_+\,  dx.
  \]
so that $C_{J_+}^+=C_{J_+}$ and the proof is complete.

\end{proof}

By \cite[Theorem 3.2 and Theorem 3.5]{M}, if $\Omega$ is bounded, $\partial\Omega$ is Lipschitz and $H$ fulfils the two-sided bound \eqref{pcontrol}, any of the corresponding energy critical point belongs to $C^\alpha(\overline{\Omega})$, with $\alpha\in\ ]0, 1[$ and its $C^\alpha(\overline\Omega)$ norm depends on $\|u\|_p$, $\Omega$ and the structural constants appearing the bounds for $H$ and $f$. Moreover, on each connected component of $\Omega$ either $u$ vanishes identically or it is strictly positive in $\Omega$.

A particular reaction $f$ and the corresponding energy critical points provide the notion of
{\em first positive Dirichlet eigenfunction}. More precisely, the number
\[
\lambda_{1, H}^+(\Omega)=\inf\left\{\int_\Omega H(Dv)\, dx:v\ge 0, \|v\|_p=1\right\}
\]
is called the first positive Dirichlet eigenvalue and the corresponding minimisers
are the normalised first positive Dirichlet eigenfunctions.
The latter are the non-negative energy critical points for $J$ with $f(t)=\lambda_{1, H}^+(\Omega)\, t^{p-1}$.

We finally prove the following form of the weak comparison principle,
which holds true under very mild assumptions on $H$ and $\Omega$.
It basically ensures that if $u$ solves \eqref{eq iniziale} in $\Omega$ and
$\underline{u}$ solves $-{\rm div}\, (DH(D\underline u))=0$ in $A\subseteq\Omega$,
then $\underline u\le u$ in $A$ as long as the inequality holds true on $\partial A$.
While usually this is deduced through a strong form of convexity for $H$, here we deal with possibly non-strictly convex $H$.

\begin{proposition}\label{comparison}
Let $\Omega$, $H$ and $f$ be as in the previous lemma and $u\in C_J^+$. For $A\subseteq \Omega$ open, let  $\underline{u}\in W^{1,p}(A)$ be a minimiser of
\[
\underline{u}+W^{1,p}_0(A) \ni v\mapsto   \int_A H(Dv)\, dx
\]
 such that $(\underline{u}-u)_+\in W^{1,p}_0(A)$. Then $u\ge \underline{u}$ in $A$.
 \end{proposition}

 \begin{proof}
 Fix a representative of $u$ and $\underline{u}$, noting that using the assumption $(\underline{u}-u)_+\in W^{1,p}_0(A)$ this can be done in such a way that $\{\underline{u}>u\}\subseteq A$. By the previous Lemma we can use $f_+$ instead of $f$ in the definition of $J$, hence we can assume $f\ge 0$.
By the previous Lemma it holds  $0\le u\le M_f$, with $M_f$ given in \eqref{propfM}, possibly infinite.  If $M_f$ is finite from $(\underline{u}-M_f)_+\le (\underline{u}-u)_+\in W^{1,p}_0(A)$ we infer that $\min\{\underline{u}, M_f\}=\underline{u}-(\underline{u}-M_f)_+\in \underline{u}+W^{1,p}_0(A)$, hence by the minimality of $\underline{u}$ againts the competitor $\min\{\underline{u}, M_f\}$ we find
 \[
 \int_A H(D\underline{u})\, dx\le \int_{\{\underline{u}< M_f\}} H(D\underline{u})\, dx
 \]
 therefore
 \[
 \int_{\{\underline{u}\ge M_f\}} H(D\underline{u})\, dx\le 0
 \]
 and thus $\underline{u}\le M_f$ in $A$.
 Since $W^{1,p}_0(A)\subseteq W^{1,p}_0(\Omega)$  by extending each element of the former as $0$ outside $A$, we can test the minimality of $u$ against $\max\{u, \underline{u}\}=u+(\underline{u}-u)_+\in W^{1,p}_0(\Omega)$ for the functional in \eqref{semil}. This gives
 \[
\int_\Omega  \frac{1}{p}\, H(Du)-f(u)\, u\, dx\le  \int_{\{\underline{u}> u\}} \frac{1}{p}\,H(D\underline{u})-f(u)\, \underline{u}\, dx
 + \int_{\{\underline{u}\le  u\}} \frac{1}{p}\,H(Du)-f(u)\, u\, dx
 \]
 so that, recalling \eqref{elle1},
 \beq
 \label{eqcomp}
   \int_{\{\underline{u}>  u\}} \frac{1}{p}\,H(Du)-f(u)\, u\, dx\le  \int_{\{\underline{u}>  u\}}\frac{1}{p}\,  H(D\underline{u})-f(u)\, \underline{u}\, dx.
  \eeq
 On the other hand, by the minimality of $\underline{u}$ against the competitor $\min\{u, \underline{u}\}=\underline u-(\underline{u}-u)_+\in \underline{u}+ W^{1,p}_0(A)$, we have
 \[
 \int_A H(D\underline{u})\, \le \int_{\{\underline{u}\le u\}} H(D\underline{u})\, dx+\int_{\{\underline{u}> u\}}H(Du)\, dx
 \]
 so that
 \[
 \int_{\{\underline{u}> u\}}H(D\underline{u})\, dx\le \int_{\{\underline{u}> u\}}H(Du)\, dx.
 \]
 Inserting the latter into \eqref{eqcomp} and rearranging we obtain
 \[
\int_{\{\underline{u}>  u\}}  f(u)\, (\underline{u}-u)\, dx\le  0.
  \]
 Since  $0\le u\le M_f$ and $f(t)>0$ for $t\in \ ]0, M_f[$, the previous inequality forces $f(u)\, (\underline{u}-u)=0$ a.\, e.\, on $\{\underline{u}>u\}$ and that equality holds in \eqref{eqcomp}. In particular a.\,e.\,point in $\{\underline{u}>u\}$ belongs to either $\{u=0\}$ or $\{u=M_f\}$. Being $\underline{u}\le M_f$ a.\,e., the second case cannot occur in a set of positive measure and thus $\{\underline{u}>u\}\subseteq \{u=0\}\cap A$ a.\,e..

 The equality in \eqref{eqcomp} then reads
 \[
 \int_{\{\underline{u}>  u\}} H(Du)\, dx =\int_{\{\underline{u}>  u\}} H(D\underline{u})\, dx
 \]
 and the left hand side vanishes since $Du=0$ a.\,e.\,on $\{u=0\}$. Moreover, still from $\{\underline{u}>u\}\subseteq \{u=0\}\cap A$, we see that $\underline{u}=(\underline{u}-u)_+$ a.\,e.\,on $A$, hence
 \[
0=\int_{\{\underline{u}>  u\}} H(D\underline{u})\, dx=\int_A H(D(\underline{u}-u)_+)\, dx
\]
which implies that $\underline{u}\le u$ a.\,e.\,in $A$ by \eqref{pcontrol2}.
\end{proof}

\subsection{Existence and uniqueness}

The existence of non-negative energy critical points can be variationally characterised for $H$ being positively $p$-homogeneous and $f$ fulfilling \eqref{hypf}.
Indeed, in this case  \cite[Theorem 5.3, point 1]{M}  provides
\beq
\label{varchar}
C_J^+={\rm Argmin}\,  J
\eeq
where, by \eqref{proCj+}, the minimisation problem on the right can be equivalently settled on either $\big(W^{1,p}_0(\Omega)\big)_+$ or $W^{1,p}_0(\Omega)$.

 Necessary and sufficient conditions on $f$ for the existence of positive critical points under assumption \eqref{hypf} have been derived for $\partial\Omega$ of class $C^2$ and $H(z)=|z|^p$ (see \cite[Proposition 3.8]{BoMoSq1}) through the Hopf Lemma. Here we consider the case of possibly non-smooth $H$ and $\Omega$, so that the Hopf lemma does not hold.

 \begin{proposition}\label{proBO}
 Let $\Omega\subseteq \R^N$ be an open, connected set with finite measure, $H\in C^0(\R^N)$ be convex, positively $p$-homogeneous and vanishing only at the origin and $f\in C^0(\R)$ be even and such that \eqref{hypf} holds true. Then either $C_J^+\setminus\{0\}$ consists of first positive Dirichlet eigenfunctions or
 \beq
 \label{proBOeq}
 C_J^+\setminus\{0\}\ne \emptyset \quad \iff \quad \lim_{t\to+\infty} \frac{f(t)}{t^{p-1}}<\lambda_{1, H}^+(\Omega)<\lim_{t\to 0^+}\frac{f(t)}{t^{p-1}}.
\eeq
\end{proposition}

\begin{proof}
Suppose that
\[
\mu_\infty:= \lim_{t\to+\infty} \frac{f(t)}{t^{p-1}}<\lambda_{1, H}^+(\Omega)<\lim_{t\to 0^+}\frac{f(t)}{t^{p-1}}=:\mu_0
\]
holds true. Then for a fixed $\eps\in \ ]0, \lambda_{1, H}^+(\Omega)-\mu_\infty[$ there exists $t_0>0$ such that
\[
F(t)\le (\lambda_{1, H}^+(\Omega)-\eps)\, \frac{t^p}{p}\qquad \text{for $t>t_0$}
\]
hence for any $v\in \big(W^{1,p}_0(\Omega)\big)_+$ it holds
\[
J(v)\ge \int_\Omega \frac{1}{p}\, H(Dv)\, dx-\int_{v\le t_0} F(v)\, dx-(\lambda_{1, H}^+(\Omega)-\eps)\int_\Omega \frac{v^p}{p}\, dx
\]
so that by the definition of $\lambda_{1, H}^+(\Omega)$
\[
J(v)\ge \frac{\eps}{p\, \lambda_{1, H}^+(\Omega)} \int_\Omega H(Dv)\, dx-\sup_{[0,  t_0]} F\, |\Omega|
\]
which implies coercivity of $J$ on $\big(W^{1,p}_0(\Omega))_+$ thanks to \eqref{pcontrol} (which still holds true under the present, weaker assumption on $H$). Therefore $J$ admits a minimiser $u\in C_J^+(\Omega)$. To prove that $u\ne 0$, note that from $\mu_0>\lambda_{1, H}^+(\Omega)$ we infer that for some positive $\eps$ and $t_1$
\[
F(t)\ge (\lambda_{1, H}^+(\Omega)+\eps) \, \frac{t^p}{p}\qquad \text{for $t\in [0, t_1]$}.
\]
Therefore, if $w$ is a first positive normalised Dirichlet eigenfunction (which is bounded), for sufficiently small $t>0$  we have
\[
J(t\, w)\le \frac{t^p}{p}\int_\Omega H(Dw) - (\lambda_{1, H}^+(\Omega)+\eps)\, w^p\, dx=-\eps\, \frac{t^p}{p}<0.
\]
Thus $J(u)<0$ and $u\in C_J^+\setminus\{0\}$.

Vice-versa, suppose that $u\in C_J^+\setminus\{0\}$ and that $u$ is not a first positive Dirichlet eigenfunction. By \cite[Corollary 2.7]{M} it holds
\[
\int_\Omega H(Du)\, dx=\int_\Omega f(u)\, u\, dx
\]
so by the definition of $\lambda_{1, H}^+(\Omega)$ and \eqref{hypf}
\[
\lambda_{1, H}^+(\Omega)\, \int_\Omega u^p\, dx\le \int_\Omega H(Du)\, dx=\int_\Omega f(u)\, u\, dx=\int_\Omega \frac{f(u)}{u^{p-1}}\, u^p\, dx\le \mu_0\, \int_\Omega u^p\, dx,
\]
proving that $\mu_0\ge \lambda_{1, H}^+(\Omega)$ and furthermore that $\mu_0>\lambda_{1, H}^+(\Omega)$, since otherwise the previous inequalities are all equalities, forcing
\[
f(t)=\lambda_{1, H}^+(\Omega)\, t^p\qquad\text{ on $u(\Omega)$},
\]
 i.\,e.\,that $u$ is a first positive Dirichlet eigenfunction.
To prove that $\mu_\infty< \lambda_{1, H}^+(\Omega)$, recall that for any bounded $v\in \big(W^{1,p}_0(\Omega)\big)_+$ it holds (see \cite[ eq. (5.7)]{M})
\[
\int_\Omega H(v)\, dx\ge \int_\Omega \frac{f(u)}{u^{p-1}}\, v^p\, dx
\]
and in particular the integrand on the right is in $L^1(\Omega)$.
Choosing $v$ to be a first positive Dirichlet eigenfunction   in the previous inequality,  we obtain by \eqref{hypf}
\[
\lambda_{1, H}^+(\Omega)\, \int_\Omega v^p\, dx=\int_\Omega H(v)\, dx\ge \int_\Omega \frac{f(u)}{u^{p-1}}\, v^p\, dx\ge \mu_\infty\, \int_\Omega v^p\, dx
\]
so that $\mu_\infty\le \lambda_{1, H}^+(\Omega)$ and, as before, equality holds if and only if $u$ is a first positive Dirichlet eigenfunction.

\end{proof}

Regarding uniqueness, we recall the following result from \cite[Theorem 5.3]{M}.

\begin{proposition}\label{prouniqueness}
Under the assumptions of the previous proposition, let $u\in C_J^+\setminus\{0\}$. Then either $u$ is a first positive Dirichlet eigenfunction or $u$ is the unique positive energy critical point for $J$ in the following cases:
\begin{enumerate}
\item
$H$ is strictly convex
\item
$t\mapsto f(t)/t^{p-1}$ is strictly decreasing on $u(\Omega)$.
\end{enumerate}
Moreover, if $H$ is strictly convex the first positive Dirichlet eigenvalue is simple, meaning that any other first positive Dirichlet eigenfunction is a positive  scalar multiple of $u$.

\end{proposition}

\section{Regularised problems}

In this section we construct regular problems approximating \eqref{eq iniziale} and derive the relevant regularity properties of the solutions, together with their boundary behaviour.

\subsection{Approximation scheme}

\begin{lemma}\label{lemmaquad}
Let $G\in C^\infty(\R^N)$ be such that
\[
G(0)=0, \qquad D^2G(z)\ge \lambda\, {\rm Id}\quad \forall z\in \R^N,
\]
 for fixed $\lambda>0$ and set
\[
\f(z)=\inf\left\{t>0: G(z/t)\le 1\right\}.
\]
Then $\f\in C^\infty(\R^N\setminus\{0\})$, is positively $1$-homogeneous and
\begin{equation}
\label{strictell}
\left(D^2\f(z)\, v, v\right)\ge \widehat{\lambda}\, |v|^2, \qquad \forall z, v \ \text{such that $\Phi(z)=1$,
$\left(D\Phi(z), v\right)=0$}
\end{equation}
for $\widehat{\lambda}>0$ depending on $G$ and $\lambda$.
\end{lemma}

\begin{proof}
The assumptions on $G$ ensure that the latter is strongly convex,
so that $\f$, being the Minkowski functional of $\{G\le 1\}$,
is automatically  $C^\infty(\R^N\setminus\{0\})$ and positively $1$-homogeneous.
By construction it holds $\{G=1\}=\{\Phi=1\}$ and
\[
G\left(z/\f(z)\right)=1\qquad \forall z\ne 0
\]
which, differentiated, gives
\[
\frac{DG(z/\f(z))}{\f(z)} -\frac{D\f(z)\, \left(DG(z/\f(z)), z\right)}{\f^2(z)}=0,
\]
or
\beq
\label{MG}
\f(z)\, DG(z/\f(z))=D\f(z)\left(DG(z/\f(z)), z\right).
\eeq
Differentiating once more, we obtain
\[
\begin{split}
&D\f(z)\otimes DG(z/\f(z))+ D^2G(z/\f(z))\, \left({\rm Id}-\frac{z\otimes D\f(z)}{\f(z)}\right)=
\left(DG(z/\f(z)), z\right)\, D^2\f(z)\\
&\quad +D\f(z)\otimes \left(DG(z/\f(z))+z\, D^2G(z/\f(z))\left(\frac{{\rm Id}}{\f(z)}-\frac{z\otimes D\f(z)}{\f^2(z)}\right)\right).
\end{split}
\]
For  $\f(z)=1$  we thus have
\[
\begin{split}
&D\f(z)\otimes DG(z)+ D^2G(z)\, \left({\rm Id}-z\otimes D\f(z)\right)=\left(DG(z), z\right)\, D^2\f(z)\\
&\quad +D\f(z)\otimes \left(DG(z)+z\, D^2G(z)\left({\rm Id}-z\otimes D\f(z)\right)\right).
\end{split}
\]
For such $z$'s, if $v$ obeys   $\left(D\f(z), v\right)=0$, we have
\[
D^2G(z)\, v=\left(DG(z), z\right)\, D^2\f(z)\, v+D\f(z)\, \left( z\, D^2G(z),  v\right)
\]
and taking the scalar product with $v$ provides by $(D\f(z), v)=0$
\[
\left(D^2G(z)\, v, v\right)=\left(DG(z), z\right)\, \left(D^2\f(z)\, v, v\right).
\]
The claim is thus proved with
\[
\widehat{\lambda}=\frac{\lambda}{\sup_{\{G=1\}}\left(DG(z), z\right)}.
\]
\end{proof}

\begin{proposition}\label{approuno}
Let $\Omega\subseteq \R^N$ be open, connected and with finite measure. Suppose that
\begin{enumerate}
\item
$H:\R^N\to [0, +\infty[$ is convex, positively $p$-homogeneous vanishes only at the origin
\item
$f\in C^0(\R)$ is even and fulfils \eqref{hypf}
\end{enumerate}
and  that either the convexity of $H$ or the monotonicity of $\R_+\ni t\mapsto f(t)/t^{p-1}$ are strict.
If $u\in  C_J^+\setminus\{0\}$  is not a first positive Dirichlet eigenfunction, there exists a sequence of convex, positively $p$-homogeneous  $H_n\in C^1(\R^N)\cap C^\infty(\R^N\setminus\{0\})$ such that
\beq
\label{stimeH}
\lambda_n\, |v|^2\, |z|^{p-2}\le \left(D^2H_n(z)\, v, v\right)  \le \Lambda_n\, |v|^2\, |z|^{p-2}
\eeq
for $0<\lambda_n\le \Lambda_n$ and corresponding $u_n\in C_{J_n}^+$ with
\[
 J_n(v)=\int_\Omega\frac{1}{p}\,  H_n(Dv)-F(v)\, dx
\]
such that $u_n\rightharpoonup u$ in $W^{1, p}_0(\Omega)$.
\end{proposition}

\begin{proof}

 Fix an even $\varphi\in C^\infty_c(B_1; [0, \infty[)$ such that $\|\varphi\|_1=1$ and set, for a sequence $\eps_n\downarrow 0$, $\varphi_n(z)=\eps_n^{-N}\varphi(z/\eps_n)$.
By Jensen inequality it holds $\varphi_n*H\ge H$ and for each $n$ the function $\varphi_n*H$ is smooth and convex. We can then set
\[
G_n(z)=\varphi_n*H(z)+\eps_n\, \frac{|z|^2}{2}
\]
so that each $G_n$ is convex as well and it holds
\beq
\label{appr0}
G_n\ge H, \qquad  G_n\to H \quad \text{in $C^0_{\rm loc}(\R^N)$}.
\eeq
Define the convex sets
\[
K_n=\{ z: G_n(z)\le 1\}\subseteq \{H\le 1\}=K.
\]
Thanks to \eqref{appr0} and \eqref{pcontrol} there exists $\delta>0$ such that for sufficiently large $n$ it holds
\beq
\label{Bdelta}
B_\delta(0)\subseteq K_n
\eeq
 so that for such $n$'s we can then let $\f_n$ be the
Minkowski functional of $K_n$
and similarly define $\f$ as the Minkowski functional of $K:=\{H\le 1\}$ (notice that then $H=\f^p$).
Being $K_n\subseteq K$ it holds $\f_n\ge \Phi$ and from \eqref{Bdelta} we infer
\beq
\label{ucontrol0}
\Phi_n(z) \le |z|/\delta
\eeq
for all sufficiently large $n$.
Given $z\ne 0$, from
\[
1=G_n(z/\Phi_n(z))
\]
and the local uniform convergence of $G_n$ to $H$, we infer that any limit point $\mu\in \R$ of the bounded sequence $(\Phi_n(z))$ fulfils $H(z/\mu)=1$, i.\,e.\,$\mu=\Phi(z)$. Therefore,  a sub-subsequence argument ensures that $\f_n\to \f$ point-wise and thus, being each $\f_n$,
as well as $\f$, convex and finite, locally uniformly on $\R^N$.
Finally set
\[
H_n(z)=\f_n^p(z),
\]
which, by \eqref{ucontrol0}, satisfies
\beq
\label{ucontrol}
H_n(z)\le \frac{|z|^p}{\delta^p}\qquad \forall z\in \R^N
\eeq
for any sufficiently large $n$.
By construction, each $H_n$ is smooth on $\R^N\setminus\{0\}$,
strictly convex and positively $p$-homogeneous.
Since  $D^2G_n\ge \eps_n\, {\rm Id}$, by Lemma \ref{lemmaquad} $\f_n$ has strongly convex unit ball in the sense that \eqref{strictell} holds true,
so that \eqref{stimeH} follows from \cite[Proposition 3.1]{CoFaVa2}.
Moreover, by the previous analysis of the sequence $\f_n$, it holds
\[
H_n\ge H, \qquad H_n\to H \quad \text{in $C^0_{\rm loc}(\R^N)$}
\]
so that in particular $J_n\ge J$.
By Proposition \ref{proBO}, \eqref{proBOeq} holds true,
hence $J_n$ has a nontrivial non-negative  energy critical point $u_n$.
By the variational characterisation \eqref{varchar} and Proposition \ref{prouniqueness},
$u_n$ is the unique minimiser for $J_n$ on $\big(W^{1,p}_0(\Omega)\big)_+$, hence in particular $J_n(u_n)\le 0$.
By \eqref{proBOeq}, for a suitable $\theta\in \ ]0, 1[$ there exists $L>0$ such that
\[
f(t)\le \theta\, \lambda_{1, H}^+(\Omega)t^{p-1}\qquad \text{for $t>L$},
\]
so that for $t>0$ it holds
\[
F(t)\le \theta\, \lambda_{1, H}^+(\Omega)\frac{t^p}{p}+ L\, \sup_{[0, L]}f.
\]
By $H_n\ge H$, the definition of $\lambda^+_{1, H}(\Omega)$ and \eqref{pcontrol}  we thus infer
\[
\begin{split}
0\ge J_n(u_n)&\ge
 J(u_n)\ge
 \frac{1}{p}\, \int_\Omega H(Du_n)\, dx-\frac{\theta\, \lambda_{1, H}^+(\Omega)}{p}\,\int_{\Omega} u_n^p\, dx- C\, |\Omega|\\
&\ge \frac{1-\theta}{p}\, \int_\Omega H(Du_n)\, dx-C\, |\Omega|\ge  \frac{1-\theta}{C\, p}\, \int_\Omega |Du_n|^p\, dx-C\, |\Omega|.
\end{split}
\]
 Therefore $(u_n)$ is bounded in $W^{1,p}_0(\Omega)$.
 Suppose, up to subsequences, that $u_n\rightharpoonup \bar u\in \big(W^{1,p}_0(\Omega)\big)_+$.
 From the lower semicontinuity of $J$, $J_n\ge J$ and the minimality of $u_n$ we get
\beq
\label{gammal}
J(\bar u)\le \varliminf_n J(u_n)\le \varliminf_n J_n(u_n)\le \varliminf_n J_n(u).
\eeq
Finally,  by  \eqref{ucontrol} and dominated convergence
\[
\lim_n \int_\Omega H_n(Du)\, dx=\int_\Omega H(Du)\, dx
\]
so that
\[
J(\bar u)\le \varliminf_n J_n(u)=J(u).
\]
Since by \eqref{varchar} $u$ is a minimiser for  $J$, so is $\bar u$,
and using again Proposition \ref{prouniqueness} grants $\bar u = u$.
\end{proof}

We also provide a similar approximation scheme for first Dirichlet positive eigenfunctions.

\begin{proposition}\label{approdue}
Let $\Omega\subseteq \R^N$ be open, connected and with finite measure and
$H:\R^N\to [0, +\infty[$ be strictly convex and positively $p$-homogeneous. If $u$ is a first positive Dirichlet eigenfunction for $\lambda_{1, H}^+(\Omega)$,
there exists a sequence of  convex, positively $p$-homogeneous  $H_n\in C^1(\R^N)\cap C^\infty(\R^N\setminus\{0\})$  obeying \eqref{stimeH} and corresponding first positive Dirichlet eigenfunctions $u_n$ such that
 $u_n\rightharpoonup u$ in $W^{1, p}_0(\Omega)$.
\end{proposition}

\begin{proof}
By Proposition \ref{prouniqueness} the strict convexity of $H$ ensures that the eigenvalue $\lambda_{1, H}^+(\Omega)$ is simple, so we can normalise and suppose that $\|u\|_p=1$. For $H_n:\R^N\to \R$ defined as in the previous proof, consider the first positive Dirichlet eigenfunctions $u_n$ associated to $H_n$, normalised with unitary $L^p$ norm. Fix $v\in C^\infty_c(\Omega)$ such that $\|v\|_p=1$ and recall that $H_n\to H$ in $C^0_{\rm loc}(\R^N)$, hence
\[
\lim_n\int_\Omega H_n(Dv)\, dx=\int_\Omega H(Dv)\, dx
\]
Since by the definition of $u_n$ it holds
\[
\int_\Omega H_n(Dv)\, dx\ge \int_\Omega H_n(Du_n)\, dx\ge \frac{1}{C}\, \int_\Omega |Du_n|^p\, dx
\]
hence $(u_n)$ is bounded in $W^{1,p}_0(\Omega)$ and we can suppose that $u_n\rightharpoonup \bar u$ and $\|\bar u\|_p=1$ by the compactness of $W^{1,p}_0(\Omega)\hookrightarrow L^p(\Omega)$. The chain of inequalities \eqref{gammal} still holds true, proving that $\bar u$ is a first positive Dirichlet eigenfunction and then that $\bar u=u$ by the simplicity of $\lambda_{1, H}^+(\Omega)$ stated in Proposition \ref{prouniqueness}.
\end{proof}

In \cite{CaRiSc} is proved an Hopf Lemma for a class of anisotropic
operators, where the function giving the anisotropy is even.
In the sequel we extend the result to the case of a not even anisotropy.

\subsection{Hopf Lemma}\label{SHL}

Let $\f\in C^\infty(\R^N\setminus\{0\})$ be convex, positively $1$-homogeneous and such that
\beq
\label{proM}
\{\f\le 1\} \qquad \text{is bounded and strongly convex}.
\eeq
The {\em polar} of $\f$ is
\[
\f^\circ(z)=\sup\left\{(\xi, z): \f(\xi)\le 1\right\}
\]
and $\f^\circ$ has the same regularity of $\f$ and still fulfils \eqref{proM}
(see for instance \cite{Ro}).

Moreover, for any $x\ne 0$ it holds \cite{BePa}:
\begin{equation}\label{propr finsler1}
\f(D\f^\circ(x))\equiv 1\equiv \f^\circ(D\f(x))\,,
\end{equation}
\begin{equation}\label{propr finsler2}
\f(x)\, D\f^\circ(D\f (x))=x=\f^\circ(x)\, D\f(D \f^\circ (x))\,.
\end{equation}

\begin{lemma}
For $\f$ as given and $r>0$ set
\beq
\label{defAr}
A_r=\{x\in \R^N: r>\check \f^\circ(x)>r/2\},
\eeq
where
\[
\check \f^\circ(x) =\f^\circ(-x).
\]
For any $m>0$ there exists $\underline{u}\in C^2(\overline{A_r})$ such that
\[
\begin{cases}
{\rm div}\, (\f^{p-1}(D\underline{u})\, D\f(D\underline{u}))=0&\text{ in $A_r$}\\
\underline{u}=m&\text{on $\{\check \f^\circ=r/2\}$}\\
\underline{u}=0, \quad \partial_n \underline{u}>0&\text{on $\{\check \f^\circ=r\}$}
\end{cases}
\]
where $n$ is the inner normal to $\{\check \f^\circ\le r\}$.
\end{lemma}

\begin{proof}
We choose $\underline{u}(x)=w(\check \f^\circ(x))$ for a suitable, decreasing $w\in C^{2}([r/2, r])$.
Since $D\check \f^\circ(x)=-D\f^\circ(-x)$, we have, for $x\ne 0$ (which can be assumed henceforth as $0\notin A_r$)
\beq
\label{DcHc}
D\underline{u}(x)=w'(\check \f^\circ(x))\, D\check \f^\circ(x)=-w'(\check \f^\circ(x))\, D\f^\circ(-x).
\eeq
Therefore,  as we are assuming $w'<0$, \eqref{propr finsler1} and \eqref{propr finsler2}
and the $0$-positive homogeneity of $D\check \f^\circ$ give
\[
\f(D\underline{u}(x))=-w'(\check \f^\circ(x))\, \f(D\f^\circ(-x))=-w'(\check \f^\circ(x)).
\]
\[
D\f(D\underline{u}(x))=D\f(D\f^\circ(-x))=\frac{-x}{\f^\circ(-x)}=-\frac{x}{\check \f^\circ(x)}
\]
so that
\[
\begin{split}
-{\rm div}\, \big(\f^{p-1}(D\underline{u}(x))\, D\f(D\underline{u}(x))\big)&=
 {\rm div}\, \left(x\, \frac{(-w'(\check \f^\circ(x))^{p-1}}{\check \f^\circ(x)}\right)\\
&=\left(x, D\frac{(-w'(\check \f^\circ(x))^{p-1}}{\check \f^\circ(x)}\right)+
N\frac{(-w'(\check \f^\circ(x))^{p-1}}{\check \f^\circ(x)}
\end{split}
\]
We compute
\[
\begin{split}
D\frac{(-w'(\check \f^\circ(x))^{p-1}}{\check \f^\circ(x)}=
 &-\frac{(p-1)\, (-w'(\check \f^\circ(x))^{p-2}\, w''(\check \f^\circ(x))\, D\check \f^\circ(x)}{\check \f^\circ(x)}\\
&\quad -\frac{(-w'(\check \f^\circ(x))^{p-1}\, D\check \f^\circ(x)}{(\check \f^\circ(x))^2}
\end{split}
\]
and note that, by the $1$-positive homogeneity of $\check \f^\circ$
\[
\left(x, D\check \f^\circ(x)\right )=\check \f^\circ(x),
\]
 hence
\[
{\rm div}\, \big(\f^{p-1}(D\underline{u})\, D\f(D\underline{u})\big)=
(p-1)\, (-w'(\check \f^\circ)^{p-2}\, w''(\check \f^\circ) +(N-1)\, \frac{(-w'(\check \f^\circ)^{p-1}}{\check \f^\circ}.
\]
It thus suffices to choose a decreasing $w$ so that
\[
(p-1)\, (-w'(t))^{p-2}\, w''(t) +(N-1)\, \frac{(-w'(t)^{p-1}}{t}=0
\]
which is equivalent to
\[
\left( (-w'(t))^{p-1}\, t^{N-1}\right)'=0.
\]
All the strictly decreasing solutions on $\R_+$ of the latter ODE are given by
\[
w(t)=
\begin{cases}
\frac{A}{p-N}\, t^{(p-N)/(p-1)}+B&\text{if $p\ne N$}\\
A\, \log t+B&\text{if $p=N$}
\end{cases}
\]
for  arbitrary $A<0, B\in \R$, and if $m>0$ it is readily verified that we can indeed choose $A<0$, $B\in \R$ in such a way  that
\[
w(r/2)=m, \qquad w(r)=0
\]
and therefore also $w'<0$.
For such a choice, the corresponding $\underline{u}=w(\check \f^\circ)$
fulfils all the requirements, as the interior normal to $\{\check \f^\circ\le r\}$ is $-D\check \f^\circ/|D\check \f^\circ|$
and by \eqref{DcHc}
\[
\partial_n \underline{u}=
-\left(\frac{D\check \f^\circ}{|D\check \f^\circ|}, w'(\check \f^\circ) \, D\check \f^\circ\right)=
-w'(\check \f^\circ) \, |D\check \f^\circ|>0
\]
\end{proof}

\begin{proposition}
Suppose $\Omega$ is bounded and connected with $C^{2}$ boundary,
$H\in C^1(\R^N)\cap C^\infty(\R^N\setminus\{0\})$ is positively $p$-homogeneous and  fulfils \eqref{stimeH},
and $f\in C^0(\R, [0, +\infty[)$ obeys \eqref{growthf}.
Then any critical point $u$  for $J$ is $C^{1,\alpha}(\overline{\Omega})$, $u>0$ in $\Omega$ and
\beq
\label{hopf}
\frac{\partial u}{\partial n}>0 \qquad \text{on $\partial\Omega$}
\eeq
where $n$ is the interior normal to $\partial\Omega$.
\end{proposition}

\begin{proof}
Any critical point for $J$ (which, under the stated assumption is a $C^1$ functional) is a weak solution of \eqref{eq iniziale} with non-negative and subcritical left hand side.
The boundedness and positivity of $u$ has already been discussed and
its regularity up to the boundary follows from \eqref{stimeH} and \cite{Lie}.
Let $\f$ be the Minkowski functional of $\{H\le 1\}$, so that $H=\f^p$,
and let $A_r$ be given in \eqref{defAr}.
Since $\partial\Omega$ is $C^2$ and $\check \f^\circ\in C^2(\R^N\setminus\{0\})$,
for any $x_0\in \partial\Omega$ there exists $x_1\in \Omega$ and $r>0$ such that
\[
x_1+A_r\subseteq \Omega, \qquad \left(x_1+\overline{A_r}\right)\cap \partial\Omega=\{x_0\}.
\]
Let
\[
m=\inf\{u(x):  \check \f^\circ(x-x_1)=r/2\}>0
\]
and choose the corresponding $\underline{u}$ given in the previous Lemma.
The weak comparison principle in Proposition \ref{comparison}
ensures that $u\ge \underline{u}$, which in turn implies \eqref{hopf}.
\end{proof}

\section{Proof of the main result}

$\bullet$ {\em Step 1}

Given $u\in C_J^+$, by Lemma \ref{lemmaCj} we can assume that $F> 0$ on $]0, +\infty[$. Consider the sequence $(u_n)$ given in Propositions \ref{approuno} and \ref{approdue}, solving the corresponding regularised problems.  Suppose we can prove that $c_{\varphi(u_n)}\le 0$. Since $u_n\to u$ almost everywhere in $\Omega$ , so does $\varphi(u_n)$ to $\varphi(u)$, hence being $u$ (and thus $\varphi(u)$) continuous in $\Omega$, we will have $c_{\varphi(u)}\le 0$.

Therefore we can assume that $H\in C^1(\R^N)\cap C^\infty(\R^N\setminus\{0\})$ fulfils \eqref{stimeH2} for some given $0<\lambda\le \Lambda$ (and is therefore strictly convex).
Moreover, proceeding as in \cite[Section 4.1]{BoMoSq1} (with obvious modifications in the case $u$ is a first positive Dirichlet eigenfunction), we can assume that $\Omega$ is strongly convex with smooth boundary. In particular, the strong minimum principle and the Hopf Lemma   apply, so we can choose $\eta>0$, $\beta\in \ ]0, 1[$ such that $u\in   C^{1,\beta}(\overline{\Omega})$ and furthermore
\beq
\label{etaunodue}
 \inf_{\Omega_{\eta}} u>0,\qquad\frac{\partial u}{\partial n}>0 \text{ on $\partial\Omega$},\qquad   \inf_{\Omega\setminus \Omega_{\eta}} |Du|>0.
\eeq
We aim at proving that for any given $\delta \in \ ]0, \eta[$ (which we'll assume henceforth),
\beq
\label{claimfinale}
c_{\varphi(u)}\le 0 \qquad \text{in $\Omega_{\delta/2}\times\Omega_{\delta/2}\times [0, 1]$},
\eeq
(see \eqref{defsdelta} for $\Omega_\delta$), which will prove the theorem. In doing so, we can also assume that $\delta$ is so small that $\partial\Omega_{\delta/2}$ is strongly convex and smooth.

$\bullet$ {\em Step 2}

Denoting  $a^{2/p}=(a^2)^{1/p}$ for any $a\in \R$, we define the family of integrands (recall that $F$ is oddly extended to $\R$)
\[
G_\eps(t, z)= \frac{1}{p}\left[\eps\, F(t)^{\frac{2}{p}}+H^{\frac{2}{p}}(z)\right]^{\frac{p}{2}} -F(t)
\]
and corresponding auxiliary functionals
\begin{equation}\label{afunc}
I_\eps(u)=
\frac 1 p\io
\left[
\eps\,  F(u)^{\frac{2}{ p}}+H^{\frac{2}{p}}(D u)
\right]^{\frac p 2}dx
-\io F(u)dx
\end{equation}
If $u$ is not a first positive Dirichlet eigenfunction,  by \cite[Lemma 4.1]{BoMoSq1}, for any sufficiently small $\eps$ problem
\[
\inf \left\{ I_\eps(u): u\in W^{1,p}_0(\Omega)\right\}
\]
admits a minimiser $u_\eps$ with the property that
\beq
\label{cbeta}
u_\eps\to u\qquad \text{in $C^{\beta}(\overline{\Omega})$}.
\eeq
Indeed, the proof of \cite[Lemma 4.1]{BoMoSq1} can be repeated {\em verbatim} when $u$ is not a first positive Dirichlet eigenfunction, since only the uniqueness of the minimiser $u$ of $J$ is used therein and the latter is granted by the strict convexity of $H$ (due to the previous point) and Lemma \ref{lemmaCj}, \eqref{varchar} and Proposition \ref{prouniqueness}. When $u$ is a normalised  first positive Dirichlet eigenfunction we instead consider
\[
\inf \left\{ I_\eps(u): u\in W^{1,p}_0(\Omega), u\ge 0, \|u\|_p=1\right\}
\]
and proceed in the same way, using again Proposition \ref{prouniqueness} to ensure weak convergence of $u_\eps$ to $u$, as well as \eqref{cbeta} by uniform $C^{\beta}(\overline{\Omega})$ bounds.

$\bullet$ {\em Step 3}

We now improve the convergence of $u_\eps\to u$ beyond the $C^\beta$ level.
Any  $u_\eps$ defined above satisfies weakly the Euler-Lagrange equation for $I_\eps$,
\[
-{\rm div}\, (D_z G_\eps(u_\eps, Du_\eps))+\partial_u G_\eps(u_\eps, Du_\eps)=0
\]
which is more explicitly computed as
\beq\label{eappro}
\begin{split}
-\Div&\left(
\big(
\eps F(u_\eps)^{\frac 2 p}+
H^{\frac{2}{p}}(\n u_\eps)
\big)^{\frac{p-2}{2}}\n \frac{H^\frac{2}{p}}{2}(\n u_\eps)
\right)\\
&\quad = f(u_\eps)\left[1-
\frac \eps  p\, \big(
\eps \, F(u_\eps)^{\frac 2 p}+
H^\frac{2}{p}(\n u_\eps)
\big)^{\frac{p-2}{2}}
F(u_\eps)^{\frac{2-p}{p}}\right].
\end{split}
\eeq

Note that \eqref{cbeta} and the positivity of $u$ ensure that given any $\delta\in \ ]0, \eta[$, there exists $C>0$ such that for a sufficiently small $\eps$ (which will be assumed henceforth) it holds
\beq
\label{lbound}
\frac{1}{C}\le u_\eps\le C\qquad \text{ in $\Omega_{\delta/4}$}
\eeq
 and therefore, $F(u_\eps)$ is uniformly bounded from above and below by a positive constant.
 Thanks to Lemma \ref{lemmastell}, point 2, Tolskdorff local regularity theory \cite{To} applies , ensuring that
 \beq
 \label{c1est}
 \|u_\eps\|_{C^{1, \beta}(\overline{\Omega_{\delta/3}})}\le C
 \eeq
 for a $\beta\in\  ]0, 1[$  (possibly different from the one  in \eqref{cbeta}) depending on $\Omega, H, N, p$ and $C>0$ with the same dependencies and additionally on $\delta$ and $\|u\|_{L^\infty(\Omega)}$, but none of them depending on $\eps$. We'll assume henceforth that $\beta<\alpha$, the H\"older continuity coefficient of $f$. In particular  It follows that $u_\eps\to u$ in $C^1(\overline{\Omega_{\delta/3}})$ by Ascoli-Arzel\'a. Further regularity can be obtained noting that by \eqref{lbound} and the boundedness of $H(Du_\eps)$, the matrix $D^2_zG_\eps$ is strongly elliptic in $\Omega_{\delta/4}$ thanks to Lemma \ref{lemmastell}, point 2.  The difference quotient method yields $u_\eps\in W^{2,2}_{\rm loc}(\Omega_{\delta/4})$, with any partial derivative $w=\partial_iu_\eps$, $i=1, \dots, N$ obeying
 \[
 -{\rm div}\, \left(D_z^2 G_\eps(u_\eps, Du_\eps)\, Dw\right)={\rm div}\, \left(\partial_u D_zG_\eps(u_\eps, Du_\eps)\, w+e_i\, \partial_u G_\eps(u_\eps, Du_\eps)\right)
 \]
 weakly in $\Omega_{\delta/4}$.
 Using the $C^{1,\beta}$ regularity of $u_\eps$, we see that
 \[
 D_z^2 G_\eps(u_\eps, Du_\eps)\qquad \text{and}\qquad \partial_u D_zG_\eps(u_\eps, Du_\eps)\, w+e_i\, \partial_u G_\eps(u_\eps, Du_\eps)
 \]
 are $\beta$-H\"older continuous in $\Omega_{\delta/3}$, so that local linear regularity theory ensures $w\in C^{1, \beta}(\Omega_{\delta/3})$, i.\,e.\,$u_\eps\in C^{2,\beta}(\Omega_{\delta/3})$.
 As $\eps\to 0$, the $C^{2, \beta}(\Omega_{\delta/3})$ norm of $u_\eps$ may blow up, but by \eqref{etaunodue}, for sufficiently small $\eps$  it holds
 \[
 \inf_{\Omega\setminus\Omega_\eta}|Du_\eps|>0,
 \]
 therefore, by looking at \eqref{stell2}, we see that $D^2_zG_\eps$ is strongly elliptic in $\Omega\setminus\Omega_{\eta}$, with ellipticity constants uniformly bounded from below and above as $\eps\to 0$.  We conclude by local elliptic regularity theory that, given $\delta\in \ ]0, \eta[$, for any sufficiently small $\eps$ it holds
 \beq
 \label{c2est}
 \|u_\eps\|_{C^{2, \beta}(\overline{\Omega_{\delta/2}}\setminus \Omega_{\delta})}\le C
 \eeq
 with $C$ depending only on $\Omega, H, N, p, \delta, \eta$ and $\|u\|_\infty$, but not on $\eps$.

 $\bullet$ {\em Step 4}

Let $v_\eps=\varphi(u_\eps)$. We claim that, for any sufficiently small $\delta\in \ ]0, \eta[$ and, correspondingly, sufficiently small $\eps$,  $c_{v_\eps}$ cannot assume a positive  maximum on $\partial(\Omega_{\delta/2}\times\Omega_{\delta/2})\times[0, 1]$. Indeed, Proposition \ref{propKorevaar} applies to $u$, so that for any sufficiently sufficiently small $\delta$, $v=\varphi(u)$ fulfils  \eqref{derseconda} in $\Omega\setminus\Omega_\delta$ and   \eqref{sottopiani} for all $x_0\in \Omega\setminus\Omega_\delta$, $x\in \Omega\setminus\{x_0\}$. From the uniform bounds \eqref{c1est} and \eqref{c2est} proved in the previous step, we infer by Ascoli-Arzel\'a that
\[
u_\eps\to u\qquad \text{ in $C^1(\overline{\Omega_{\delta/2}})\cap C^2(\overline{\Omega_{\delta/2}}\setminus\Omega_\delta)$}.
\]
Since $u_\eps$  fulfils \eqref{lbound} for sufficiently small $\eps$  and since $\varphi\in C^{2, \alpha}_{\rm loc}(\R_+)$, the previous convergence holds true for $v_\eps$ as well. Note that $\partial\Omega_{\delta/2}\subseteq \Omega\setminus\Omega_\delta$, hence $v$ fulfils \eqref{sottopiani} for all $x_0\in\partial\Omega_{\delta/2}$ and $x\in \overline{\Omega_{\delta/2}}\setminus\{x_0\}$. Then
Proposition \ref{prostable}, applied to the family $v_\eps$ and $v$ on the strictly convex set $\Omega_{\delta/2}$, ensures that  for any sufficiently small $\eps>0$ \eqref{sottopiani} holds true for $v_\eps$ at all points $x_0\in \partial\Omega_{\delta/2}$, $x\in \overline{\Omega_{\delta/2}}\setminus\{x_0\}$.  Finally, Proposition \ref{propKorevaar2}, applied to such functions  $v_\eps$  on $\Omega_{\delta/2}$,  proves the claim.

$\bullet$ {\em Step 5}

For $\eps$ and $v_\eps$ as above and $M_f$ as in \eqref{defM}, we look at the equation fulfilled by $v_\eps$. The proof of the last statement of Lemma \ref{lemmaCj} still holds for the functional $I_\eps$ and its minimiser $u_\eps$, showing that
\[
M_\eps=\sup_\Omega u_\eps\le M_f.
\]
Let $\psi=\varphi^{-1}$, so that for $t\in \ ]0, M_\eps]$ it holds
\[
\varphi'(t)=F^{-\frac 1 p }(t), \qquad \psi'(s)=F^{\frac 1 p}(\psi(s)).
\]
We thus have $F(u_\eps)=F(\psi(v_\eps))$ and
\beq
\label{t02}
 H(Du_\eps)=H(\psi'(v_\eps)\, Dv_\eps)=F(\psi(v_\eps))\, H(Dv_\eps)
\eeq
where we used the positive $p$-homogeneity of $H$ in the last step. Similarly, by the positive $1$-homogeneity of $DH^{2/p}$
\[
(D H^{\frac{2}{p}})(Du_\eps)=\psi'(v_\eps)\,  (D H^{\frac{2}{p}})(Dv_\eps)=F^{\frac 1 p}(\psi(v_\eps))\, (D H^{\frac{2}{p}})(Dv_\eps).
\]
We look at equation \eqref{eappro} for $v_\eps$. For the left hand side we compute
\[
\begin{split}
{\rm div}\, &\left(
\big(
\eps F(u_\eps)^{\frac 2 p}+
H^{\frac{2}{p}}(\n u_\eps)
\big)^{\frac{p-2}{2}}\n \frac{H}{2}^\frac{2}{p}(\n u_\eps)
\right)\\
&={\rm div}\, \left( F^{1-\frac{1}{p}}(\psi(v_\eps))\, \big(\eps+H^{\frac{2}{p}}(Dv_\eps)\big)^{\frac{p-2}{2}}\, D \frac{H^{\frac{2}{p}}}{2}(Dv_\eps)\right)\\
&=\left(1-\frac{1}{p}\right) F^{-\frac 1 p}(\psi(v_\eps))\, f(\psi(v_\eps))\, \psi'(v_\eps)\, \big(\eps+H^{\frac{2}{p}}(Dv_\eps)\big)^{\frac{p-2}{2}}\, \left(D \frac{H^{\frac{2}{p}}}{2}(Dv_\eps), Dv_\eps\right)\\
&\quad + F^{1-\frac{1}{p}}(\psi(v_\eps))\,  {\rm div} \, \left(  \big(\eps+H^{\frac{2}{p}}(Dv_\eps)\big)^{\frac{p-2}{2}}\, D \frac{H^{\frac{2}{p}}}{2}(Dv_\eps)\right)\\
&=\left(1-\frac{1}{p}\right)  f(\psi(v_\eps))\, \big(\eps+H^{\frac{2}{p}}(Dv_\eps)\big)^{\frac{p-2}{2}}\, H^{\frac{2}{p}}(Dv_\eps)\\
&\quad + F^{1-\frac{1}{p}}(\psi(v_\eps))\,  {\rm div} \, \left(  \big(\eps+H^{\frac{2}{p}}(Dv_\eps)\big)^{\frac{p-2}{2}}\, D \frac{H^{\frac{2}{p}}}{2}(Dv_\eps)\right)
\end{split}
\]
where we used $\psi'(s)=F^{1/p}(\psi(s))$ and that, being $H^{2/p}$ positively $2$-homogeneous,
 \[
 \left(D H^{\frac{2}{p}}(Dv_\eps), Dv_\eps\right)=2\, H^{\frac{2}{p}}(Dv_\eps).
 \]
The right-hand side of \eqref{eappro} is, again by \eqref{t02},
\[
f(u_\eps)\left[1- \frac \eps  p\, \big( \eps \, F(u_\eps)^{\frac 2 p}+ H^\frac{2}{p}(\n u_\eps) \big)^{\frac{p-2}{2}} F(u_\eps)^{\frac{2-p}{p}}\right]=
 f(\psi(v_\eps))\left[1-\frac \eps  p\, \big(\eps +H^\frac{2}{p}(\n v_\eps)\big)^{\frac{p-2}{2}} \right].
 \]
 Hence $v_\eps$ satisfies
 \[
 \begin{split}
- F^{1-\frac{1}{p}}(\psi(v_\eps))\,  {\rm div} \, &\left(  \big(\eps+H^{\frac{2}{p}}(Dv_\eps)\big)^{\frac{p-2}{2}}\, D \frac{H^{\frac{2}{p}}}{2}(Dv_\eps)\right)=f(\psi(v_\eps))\left[1-\frac \eps  p\, \big(\eps +H^\frac{2}{p}(\n v_\eps)\big)^{\frac{p-2}{2}} \right]\\
&+\left(1-\frac{1}{p}\right)\, f(\psi(v_\eps))\, \big(\eps+H^{\frac{2}{p}}(Dv_\eps)\big)^{\frac{p-2}{2}}\, H^{\frac{2}{p}}(Dv_\eps)
\end{split}
\]
 which rewrites as
  \beq
  \label{efin}
  -{\rm div} \, \left(  DH_\eps(Dv_\eps)\right)= \frac{f(\psi(v_\eps))}{ F^{1-\frac 1 p }(\psi(v_\eps))} b_\eps(Dv_\eps)
  \eeq
  where
  \[
  \begin{split}
  H_\eps(z)&=\big(\eps+H^\frac{2}{p}(z)\big)^{\frac{p}{2}}\\
b_\eps(z)&= p+\big((p-1) \, H^{\frac 2 p}(z)-\eps\big)\big(\eps+H^{\frac{2}{p}}(z)\big)^{\frac{p-2}{2}}.
\end{split}
  \]

 $\bullet$ {\em Step 6}

  We finally exclude that, for $\eps$ and $v_\eps$ as above,  $c_{v_\eps}$ attains a positive maximum on $\Omega_{\delta/2}\times\Omega_{\delta/2}\times[0, 1]$.
 Note that
  \[
  b_\eps(z)\ge b(0)=p-\eps^{\frac p 2}
  \]
  which is positive for sufficiently small $\eps$.
  On the other hand, for $s\in v_\eps(\Omega_{\delta/2})\subseteq\  ]0,   \varphi(M_\eps)]$ it holds
  \[
   \frac{f(\psi(s))}{ F^{1-\frac 1 p }(\psi(s))}=\big(F^{\frac{1}{p}}\big)' (\psi(s))
   \]
   which is non-increasing since $\psi$ is non-decreasing and $F^{1/p}$ is concave,
   while
   \[
   \psi''(s)=\big(F^{\frac{1}{p}}(\psi(s))\big)'=\frac{1}{p}\, F^{1-\frac{1}{p}}(\psi(s))\, f(\psi(s))\, \psi'(s)=\frac{F(\psi(s))}{f(\psi(s))}
   \]
   so that
   \[
\frac{ F^{1-\frac{1}{p}}(\psi(s))}{f(\psi(s))}=\frac{\psi''(s)}{\psi'(s)}
  \]
   which is convex by Lemma \ref{lemmavarphi}, point 2.
 Finally, since $v_\eps\in C^2(\Omega_{\delta/2})$, \eqref{efin} rewrites as
 \[
 -{\rm Tr}\, \big(D^2H_\eps(Dv_\eps)\, D^2v_\eps)=\frac{f(\psi(v_\eps))}{ F^{1-\frac 1 p }(\psi(v_\eps))} b_\eps(Dv_\eps)
 \]
 and since $H(Dv_\eps)$ is bounded in $\Omega_{\delta/2}$, Lemma \ref{lemmastell}, point 2, grants the strong ellipticity of $D^2H_\eps(z)$ for $z\in Dv_\eps(\Omega_{\delta/2})$. Therefore Proposition \ref{propKennington} ensures that $c_{v_\eps}$ cannot attain a positive maximum on $\Omega_{\delta/2}\times\Omega_{\delta/2}\times [0, 1]$. By step 4 we conclude that, given a sufficiently small $\delta$, for any sufficiently small $\eps>0$, $c_{v_{\eps}}\le 0$ on $\Omega_{\delta/2}\times\Omega_{\delta/2}\times [0, 1]$ and taking the limit for $\eps\downarrow 0$, we infer \eqref{claimfinale}, proving the theorem.%

\appendix

\section{Ellipticity estimates}
In this appendix we prove strong ellipticity estimates for the auxiliary integrands constructed during the proof of Theorem \ref{teorema 1}, starting from a smooth, positively $p$-homogeneous $H$ obeying
\beq
\label{stimeH2}
\lambda\, |z|^{p-2}\, |v|^2\le \left(D^2H(z)\, v, v\right)\le \Lambda\, |z|^{p-2}\, |v|^2\qquad \forall z\in \R^N\setminus\{0\}, v\in \R^N.
\eeq
Their proofs are variants of \cite[Appendix A]{CoFaVa1}, where \eqref{stimeH2} is shown to be a consequence of the strong concavity of $\{H\le 1\}$.

\begin{lemma}\label{lemmastell}
Suppose $H\in C^1(\R^N)\cap C^2(\R^N\setminus\{0\})$ is positively $p$-homogeneous and fulfils for $0<\lambda\le \Lambda$ the ellipticity estimate \eqref{stimeH2}.

 Then
 \begin{enumerate}
 \item
 $H^{2/p}$ is strongly elliptic in the sense that there exists positive $\widehat\lambda$, $\widehat\Lambda$ depending only $H$ and $p$  such that for any $z, v\in \R^N$ it holds
\beq
\label{stell}
\widehat\lambda\, |v|^2\le \left(D^2 H^{2/p}(z)\, v, v\right)\le \widehat\Lambda\, |v|^2
\eeq
\item
For any $\theta>0$ the function
\[
H_\theta(z)=\left(\theta+H^{\frac{2}{p}}(z)\right)^{\frac{p}{2}}
\]
fulfils
\beq
\label{stell2}
\widetilde{\lambda}\, \left(\theta+H^{\frac{2}{p}}(z)\right)^{\frac{p-2}{2}}\, |v|^2\le \left(D^2 H_\theta(z)\, v, v\right)\le \widetilde{\Lambda}\, \left(\theta+H^{\frac{2}{p}}(z)\right)^{\frac{p-2}{2}}\, |v|^2
\eeq
for all $z, v\in \R^N$, where  $\widetilde{\lambda}, \widetilde{\Lambda}$ are positive numbers depending on $H$ and $p$ but not on $\theta$.
\end{enumerate}
\end{lemma}

\begin{proof}
From the positive $p$-homogeneity of $H$ we get
\beq
\label{pom}
p\, H(z)=\left(DH(z), z\right), \qquad p\, (p-1) H(z)=\left(D^2H(z)\, z, z\right)
\eeq
as well as
\beq
\label{pom2}
D^2H(z)\, z=(p-1)DH(z).
\eeq
Let
\[
n_z=\frac{DH(z)}{|DH(z)|}
\]
 be the exterior normal  to  the level sets of $H$. Then by \eqref{pom} for any $z\ne 0$
\[
 (z, n_z)= p\, \frac{H(z)}{|DH(z)|}\ge c\, |z|
 \]
 where
\[
 c=c(H, p)=\inf_{z \ne 0} p\, \frac{H(z)}{|DH(z)|\, |z|}=\inf_{|z|=1} p\, \frac{H(z)}{|DH(z)|}
 \]
 (where we used the $p$-positive homogeneity of $H$ and  of $|DH(z)|\, |z|$), which is finite and positive.
 In particular any $v\in \R^N$ can be uniquely written as
 \beq
 \label{v}
 v=k\, z+ t\qquad \text{ with $k\in \R$, $t\in \R^N$, $(t, n_z)=0$}.
 \eeq
 We clearly have
\[
 |v|^2\le 2\left(|t|^2+k^2\, |z|^2\right).
 \]
 On the other hand, by Schwartz inequality
\[
 |v|\ge |(v, n_z)|=|k|\, (z, n_z)\ge c\, |k|\, |z|,
 \]
while by triangle inequality and the latter estimate
\[
|t|\le |v|+|k|\, |z|\le \left(1+\frac{1}{c}\right) |v|.
\]
All in all, we have found  a constant $C=C(H, p)$ such that for any $z\ne 0$ and all $v\in \R^N$ decomposed as in \eqref{v}, it holds
 \beq
 \label{stimev}
 \frac{1}{C}\, (k^2\, |z|^2+|t|^2)\le |v|^2\le C\, (k^2\, |z|^2+|t|^2).
 \eeq
Decomposition \eqref{v} allows the following computations for $z\ne 0$.
Thanks to \eqref{pom} and $(DH(z), t)=0$, we have
 \beq
 \label{puno}
 (DH(z)\, v, v)^2=k^2 (DH(z), z)^2=p^2\,H^2(z)\,  k^2.
 \eeq
 On the other hand,
 \beq
 \label{pdue}
 \begin{split}
 (D^2H(z)\, v, v)&=k^2\, (D^2H(z)\, z, z)+2\, k\, (D^2H(z)\, z, t)+(D^2H(z)\, t, t)\\
 &=p\, (p-1)\, H(z)\, k^2 +(D^2H(z)\, t, t)
 \end{split}
 \eeq
 thanks to \eqref{pom}, \eqref{pom2} and again $(DH(z), t)=0$.

 With these tools at hand, let us prove assertion (1) of the Lemma. Being $H^{2/p}$ a positively $2$-homogeneous function, $D^2 H^{2/p}$ is $0$-homogeneous, hence it suffices to consider the case $z\in \{H=1\}$.
 We compute
 \[
 D^2 H^{\frac 2 p }(z)=\frac{2}{p}\, H^{\frac{2-p}{p}}(z)\left[\frac{2-p}{p} \, H^{-1}(z)\, DH(z)\otimes DH(z)+D^2H(z)\right]
 \]
 so that for a given   $z\in \{H=1\}$
 \beq
 \label{pzero}
 \left(D^2 H^{\frac 2 p }(z)\, v, v\right)=\frac{2}{p} \left[\frac{2-p}{p} \,(DH(z)\, v, v)^2+(D^2H(z)\, v, v)\right].
 \eeq
 Inserting \eqref{puno} and \eqref{pdue} in \eqref{pzero} gives, for any $z\in \{H=1\}$,
 \[
  \left(D^2 H^{\frac 2 p }(z)\, v, v\right)=2\, k^2+\frac{2}{p}\, (D^2H(z)\, t, t)
  \]
 and using \eqref{stimeH2} we obtain
 \[
2\,  k^2+\frac{2\, \lambda}{p}\, |z|^{p-2}\, |t|^2 \le\left(D^2 H^{\frac 2 p }(z)\, v, v\right)\le 2\,  k^2+\frac{2\, \Lambda}{p}\, |z|^{p-2}\, |t|^2.
\]
Since $|z|$ is uniformly bounded from above and below on $\{H=1\}$, the two-sided estimate \eqref{stimev} provides \eqref{stell} for $z\in \{H=1\}$, and thus for all $z\in \R^N$ by $0$-homogeneity.

To prove assertion (2), we set $\widehat{H}=H^{2/p}$, which is $2$ homogeneous and satisfies \eqref{stell}, so that
\[
H_\theta(z)=\left(\theta+\widehat{H}(z)\right)^{\frac{p}{2}}.
\]
A standard computation gives
\[
D^2H_\theta(z)=\frac{p}{2}\left(\theta+\widehat{H}(z)\right)^{\frac{p-2}{2}}\left[\frac{p-2}{2\, (\theta+\widehat{H}(z))} D\widehat H(z)\otimes D\widehat H(z)+D^2\widehat H(z)\right].
\]
For $z\ne 0$ and $v\in \R^N$,  we consider again the decomposition \eqref{v}, so that \eqref{puno} and \eqref{pdue} applied to $\widehat H$ (so that  $p=2$ therein), give
\[
\begin{split}
\left(D^2H_\theta(z)\, v, v\right)&=\frac{p}{2}\left(\theta+\widehat{H}(z)\right)^{\frac{p-2}{2}}\left[\frac{p-2}{2\, (\theta+\widehat{H}(z))}\, 4\, \widehat H^2(z)\, k^2 +2\, \widehat H(z)\, k^2+ \left(D^2H(z)\, t, t\right)\right]\\
&=p\left(\theta+\widehat{H}(z)\right)^{\frac{p-2}{2}} \left[\frac{\theta+(p-1)\, \widehat{H}(z)}{ \theta+\widehat{H}(z)}\,  \widehat H(z)\, k^2+\frac{1}{2}\left(D^2\widehat H(z)\, t, t\right)\right].
\end{split}
\]
Next note that
\[
a_p:=\min\{1, p-1\}=\min_{t\ge 0}\frac{\theta+(p-1)\,t}{ \theta+t}, \qquad b_p:=\max\{1, p-1\}=\max_{t\ge 0}\frac{\theta+(p-1)\,t}{ \theta+t}
\]
which are positive and independent of $\theta$, hence \eqref{stell} provides
\beq
\label{afin}
\begin{split}
p\left(\theta+\widehat{H}(z)\right)^{\frac{p-2}{2}}  \left[a_p\, \widehat H(z)\, k^2+\frac{\widehat\lambda}{2}\, |t|^2\right]&\le \left(D^2H_\theta (z)\, v, v\right)\\
&\le p\left(\theta+\widehat{H}(z)\right)^{\frac{p-2}{2}}\left[b_p\, \widehat H(z)\, k^2+\frac{\widehat\Lambda}{2}\, |t|^2\right]
\end{split}
\eeq
where $\widehat{\lambda}$ and $\widehat{\Lambda}$ are given in \eqref{stell}, hence are independent of $\theta$.
Since $\widehat H$ is positively $2$-homogeneous and vanishes only at the origin, we readily have
\[
\frac{1}{C}\, |z|^2\le \widehat H(z)\le C\, |z|^2
\]
for a constant $C=C(H, p)$, and using these inequalities in \eqref{afin} together with \eqref{stimev} provides \eqref{stell2} with the stated dependencies.
\end{proof}

\end{document}